\documentclass[11pt]{amsart}
\usepackage{color, amssymb, amsmath,amsfonts,mathrsfs}
\usepackage[shortlabels]{enumitem}

\usepackage{ mathrsfs, amsmath, amsthm, amssymb,
thm-restate, pifont}
\usepackage{textcomp}

\usepackage[colorlinks]{hyperref}
\hypersetup{linkcolor=blue, urlcolor=blue, citecolor=red}

\usepackage[capitalise]{cleveref}

\usepackage[colorlinks]{hyperref}
\hypersetup{linkcolor=blue, urlcolor=blue, citecolor=red}
\usepackage[capitalise]{cleveref}

\newcommand{\cmark}{\ding{51}}%
\newcommand{\xmark}{\ding{55}}%
\newtheorem{theorem}{Theorem}[section]
\newtheorem{proposition}[theorem]{Proposition}
\newtheorem{corollary}[theorem]{Corollary}
\newtheorem{lemma}[theorem]{Lemma}
\newtheorem{question}[theorem]{Question}

\newtheorem{remark}[theorem]{Remark}

\theoremstyle{definition}
\newtheorem{definition}[theorem]{Definition}

\newcommand{\N}{\mathbb N}
\newcommand{\Xendd}{\operatorname{End}(X, \leq)}
\newcommand{\w}{\omega}
\newcommand{\Tau}{\mathcal T}
\newcommand{\E}{\operatorname{End}}
\newcommand{\Zendd}{\operatorname{End}(\mathbb Z, \leq)}

\newcommand{\End}{\operatorname{End}(\mathbb N, <)}
\newcommand{\Endd}{\operatorname{End}(\mathbb N, \leq)}
\newcommand{\Enddi}{\operatorname{End}^{\infty}(\mathbb N, \leq)}
\newcommand{\Inj}{\operatorname{Inj}}
\newcommand{\Z}{\mathbb Z}
\newcommand{\im}{\operatorname{im}}
\newcommand{\dom}{\operatorname{dom}}
\newcommand{\id}{\operatorname{id}}
  
\newcommand{\Surj}{\operatorname{Surj}}  
  
\newcommand{\Sym}{\operatorname{Sym}}  
\newcommand{\set}[2]{\left\{#1:#2\right\}}
\newcommand{\makeset}[2]{\left\lbrace #1 \; \colon\;
 \begin{tabular}{@{}l@{}}
   #2
  \end{tabular}
  \right\rbrace}
\title{Polish topologies on endomorphism monoids of linear orders}

\author{S. Bardyla, L. Elliott}

\address{S.~Bardyla: University of Vienna, Institute of Mathematics, Vienna, Austria.}
\email{sbardyla@gmail.com}
\address{L. Elliott: University of Manchester, Department of Mathematics and Statistics,
Oxford Road, Manchester, M13 9PL, UK}
\email{luna.elliott142857@gmail.com}
\makeatletter
\@namedef{subjclassname@2020}{\textup{2020} Mathematics Subject Classification}
\makeatother

\subjclass[2020]{20M20, 46H40, 54H15, 54E52}
\keywords{Endomorphism monoids of partial orders, Polish semigroups, Zariski topology, maximal topologies, minimal topologies}

\thanks{The first named author was supported by the Austrian Science Fund FWF (Grant ESP 399).}
\thanks{The second named author was supported by the Heilbronn Institute for Mathematical
Research.}

\begin{document}

\begin{abstract}
In this paper, we investigate Polish semigroup topologies on the endomorphism monoids   $\operatorname{End}(\mathbb{N},\leq)$ and $\operatorname{End}(\mathbb{Z},\leq)$. We introduce a new structural condition, property $\mathbb{XX}$, which yields automatic continuity of Borel measurable homomorphisms between certain topological semigroups. This provides a new method for analyzing Polish semigroup topologies on monoids with small groups of units. We show that for all monoids considered, the semigroup Zariski topology coincides with the pointwise topology and is therefore the coarsest Hausdorff semigroup topology. We prove that the submonoid $\operatorname{End}^{\infty}(\mathbb{N},\leq)$ of $\operatorname{End}(\mathbb{N},\leq)$ consisting of all endomorphisms with infinite image admits a unique Polish semigroup topology, namely the pointwise topology. On the other hand, despite possessing a finest Polish semigroup topology, 
the monoids $\operatorname{End}(\mathbb{N},\leq)$ and $\operatorname{End}(\mathbb{Z},\leq)$, admit infinitely many distinct Polish semigroup topologies. Also, we show that the monoid $\operatorname{End}(\mathbb{N},<)$ admits exactly $2^{\aleph_0}$ Polish semigroup topologies and no maximal second-countable semigroup topology.  
\end{abstract}
\maketitle

\tableofcontents

\section{Introduction and main results}
In this paper, we compose functions from left to right, i.e., $(x)fg=((x)f)g$, and we let $\N$ denote the set of all non-negative integers.
Recall that the {\em full transformation monoid} $X ^ X$ on a set $X$ consists of all functions from $X$ to $X$ equipped with the operation of composition of functions. 
Let $S$ be a subsemigroup of the full transformation monoid $X^X$. Then the {\em pointwise topology on $S$} is generated by the subbasis consisting of the sets $$U_{x,y}:=\{f\in S: (x)f=y\},$$ where $x,y\in X$. 

It follows from~\cite{BGP, Gaughan} that the pointwise topology is the unique Polish group topology on the symmetric group $\Sym(\N)$, i.e., it is the unique topology which turns $\Sym(\N)$ into a separable completely metrizable topological group. Later Kechris and Rosendal~\cite{KR} generalized this result and showed that the pointwise topology is the unique separable Hausdorff group topology on $\Sym(\N)$. Similar results were discovered for other natural semigroups. For instance, it was shown in~\cite{main} that the pointwise topology is the unique Polish semigroup topology on the full transformation monoid $\N^\N$, i.e., it is the unique topology which turns $\N^\N$ into a Polish topological semigroup. Also, the monoids of partial transformations of $\N$, and continuous self-maps of the Hilbert cube $[0,1]^\N$ or the Cantor space $2^\N$ possess a unique Polish semigroup topology~\cite{main}. 

Let $\rho$ be a binary relation on a set $X$. Recall that the {\em endomorphism monoid} $\operatorname{End}(X,\rho)$ is the submonoid of $X^X$ such that for each $f\in \operatorname{End}(X,\rho)$ and $(a,b)\in\rho$ we have $((a)f,(b)f)\in\rho$. The algebraic and topological structures of endomorphism monoids are widely studied in the literature, see~\cite{Blyth, Bodirsky, BEKP, Coleman, EJMPP, Gunturkun, MP, Pinskernew, PiSc2, PiSc}. It was shown in~\cite{EJMPP} that the pointwise topology is the unique Polish semigroup topology on endomorphism monoids of the random graph, random digraph, and random partial order. 
In~\cite{PiSc} it was shown that the pointwise topology is the unique Polish semigroup topology on $\operatorname{End}(\mathbb Q,\leq)$. Also, in~\cite{Pinskernew} it was shown that the pointwise topology is a minimal Hausdorff semigroup topology on the elementary embedding monoids of various structures.

In this paper, we continue this theme by investigating Polish semigroup topologies on the endomorphism monoids $\End$, $\Endd$, and $\Zendd$. 
It is noteworthy that the aforementioned results on the uniqueness of Polish semigroup topologies on endomorphism monoids were proven using different versions of the so-called property $\mathbb X$ (introduced in~\cite{main}) applied to the group of units of the corresponding monoids. 
In particular, it was crucial that the corresponding relational structures have large automorphism groups.
Since $\End$ and $\Endd$ have trivial groups of units and the group of units of $\Zendd$ is isomorphic to the additive group of integers, the techniques which depend on a large group of units are not applicable to these monoids. 
To remedy this, we introduce a new property defined below.

%Moreover, in the case of \(\N^\N\), it was shown in \cite{} that \(\Sym(\N)\) has automatic continuity with respect to the class of second-countable groups and property \(\mathbb{X}\) enabled the analogous result to be proven for \(\N^\N\) as well given the result for \(\Sym(\N)\).

\begin{definition}\label{XX}
    We say that a semigroup \(S\) equipped with a topology has property $\mathbb{XX}$ with respect to a subset \(A\subseteq S\) if 
    \begin{enumerate}
        \item[(1)]  \(A\) is dense \(G_\delta\) in $S$;
        \item[(2)] for all \((a, s)\in A\times S\) there exist $b_0,\ldots,b_n\in S^1$ (here $S^1=S$, if $S$ is a monoid, and $S^1$ is equal to $S$ with an adjoined unit, otherwise) such that the map $t_{a,s}:S\rightarrow S$ defined by $(x)t_{a,s}=b_0xb_1\cdots xb_n$ has the following properties:
         \begin{enumerate}
    \item[(2.1)] $t_{a,s}$ is continuous;
    \item[(2.2)] $t_{a,s}{\restriction}_A: A\rightarrow S$ is an open map (here $A$ has the subspace topology inherited from $S$); 
    \item[(2.3)] \((a)t_{a, s}=s\).
   % \item for all \((a, s)\in A\times S\) there is a continuous semigroup polynomial \(t_{a, s}\) on \(S\) which is an open mapping from \(A\) to $S$ such that \((a)t_{a, s}=s\), where $A$ is equipped with the subspace topology.
    \end{enumerate}
    \end{enumerate}   
\end{definition}
%Note that trivially every  group has property $\mathbb{XX}$ with respect to itself. 
The property $\mathbb{XX}$ allows us to prove the automatic continuity of Borel measurable homomorphisms between certain topological semigroups. In particular, the following technical result (which appears later as \cref{borelcontinuous-XX}) is crucial for this paper.
% This is a useful tool for studying uniqueness of topologies on semigroups. 

\begin{theorem}
    Let \(S\) and \(T\) be topological semigroups such that
    \begin{itemize}
        \item \(S\) is Polish;
        \item \(T\) is regular and second-countable;
        \item \(S\) has property $\mathbb{XX}$ with respect to some \(A\subseteq S\).
   \end{itemize}
    In this case every Borel measurable homomorphism from \(S\) to \(T\) is continuous. 
\end{theorem}

For more on the automatic continuity of homomorphisms between topological groups and semigroups, see for instance~\cite{Banakh, main, EJMPP, Kuzne, pech2016automatic, Rosen1, Rosen2, Rosen3, Rosen4survey, Tsankov} and references therein.

In most of the aforementioned cases the unique Polish semigroup topology on a given semigroup coincides with the semigroup Zariski topology, which has purely algebraic nature and is defined below. 
 
\begin{definition}
The {\em semigroup Zariski} topology $\mathfrak{Z}(S)$
on a semigroup $S$ is generated by the subbasis consisting
of the sets $$\set{x\in S}{a_0xa_1\cdots xa_n \neq b_0xb_1\cdots xb_m},$$ where $n,m\in\N$ and $\{a_i: i\leq n\}\cup\{b_j: j\leq m\}\subseteq S$. 
\end{definition}

The following four theorems are the remaining main results of this paper. For the reader's convenience, in the statement of these theorems we also add references to the places, where the corresponding facts are proven.

\begin{theorem}\label{thm:endNl}
The following hold for topologies on the monoid \(\End\):
\begin{enumerate}[\rm(1)]
    \item The semigroup Zariski topology coincides with the pointwise topology and is thus the coarsest Hausdorff semigroup topology (\cref{lem:zariski_pointwise}).
    \item There are no maximal second-countable semigroup topologies (\cref{nosecondAC}).
    \item There are precisely \(2^{\aleph_0}\) Polish semigroup topologies (\cref{continuun many things}).
    \item There are at least \(2^{\aleph_0}\) second-countable metrizable non-Polish semigroup topologies (\cref{continuun many things}).
\end{enumerate}
\end{theorem}

\begin{theorem}\label{thm:endN}
The following hold for topologies on the monoid \(\Endd\):
\begin{enumerate}[\rm(1)]
    \item The semigroup Zariski topology coincides with the pointwise topology and is thus the coarsest Hausdorff semigroup topology (\cref{zariski_order}).
    \item There are infinitely many Polish semigroup topologies (\cref{infinite polish N}).
    \item There is a unique maximal Polish semigroup topology. Moreover this topology contains all Polish semigroup topologies on \(\Endd\)  (\cref{maxN}).
\end{enumerate}
\end{theorem}

\begin{definition}
  Let $\operatorname{End}^\infty(\N,\leq)$ denote the submonoid of $\operatorname{End}(\N,\leq)$ consisting of those maps with infinite image.   
\end{definition}

\begin{theorem}\label{thm:endNi}
The following hold for topologies on the monoid \(\Enddi\):
\begin{enumerate}[\rm(1)]
    \item The semigroup Zariski topology coincides with the pointwise topology and is thus the coarsest Hausdorff semigroup topology (\cref{magenta}).
    \item The pointwise topology is the unique Polish semigroup topology (\cref{maximum}).
\end{enumerate}
\end{theorem}

\begin{theorem}\label{thm:endZ}
The following hold for topologies on the monoid \(\operatorname{End}(\Z, \leq)\):
\begin{enumerate}[\rm(1)]
    \item The semigroup Zariski topology coincides with the pointwise topology and is thus the coarsest Hausdorff semigroup topology (\cref{zariski_order}).
    \item There are infinitely many Polish semigroup topologies (\cref{infinite polish Z}).
    \item There is a unique maximal Polish semigroup topology. Moreover this topology contains all Polish semigroup topologies on \(\operatorname{End}(\Z, \leq)\) (\cref{maxZ}).
\end{enumerate}
\end{theorem}

In Table \ref{table-the-only}, the reader can see a summary of how our results complement the existing results in the literature.

 %We also modified the columns to be more applicable to this paper.

\begin{figure}
    \centering

\begin{tabular}{ll|c|c|c|c|l}
                                &                                     &MP & AC
                                          & \# PST                               &  Z                        &Reference                                                                                          \\ \hline
    full transformations of $\N$                 & $\N ^ \N$                          & \cmark  & \cmark
                                          & 1                                    & \cmark 
                                                                       &
    \cite{main}
    \\
    binary relations on $\N$                     & $B_{\N}$                             & \xmark& \cmark
                                          & 0                                     & \cmark        &
    \cite{main}
    \\
    partial transformations of $\N$               & $P_{\N}$                            & \cmark & \cmark
                                          & 1                                     & \cmark                            &
    \cite{main}
    \\
    partial bijections of $\N$                   & $I_{\N}$                            & \cmark & \cmark
                                          & $\aleph_0$                               & \cmark     
                                                                              &
    \cite{main}+\cite{BEMP}
    \\
    injective transformations of $\N$            & $\Inj(\N)$                         & \cmark  & \cmark
                                          & $\geq\aleph_0$                             & \cmark 
                                                                              &
    \cite{main}
    \\
    surjective transformations of $\N$           & $\Surj(\N)$   &?                       & ?
                                          & $\geq\aleph_0$                             & ?      
                                                                              &
    \cite{main}
    \\
    continuous self-maps on $[0, 1] ^ \N$ & $C([0, 1] ^ \N)$ & \cmark                    & ?
                                          & 1                                          & ?                & \cite{main}                                   \\
    continuous self-maps on $2 ^ \N$      & $C(2 ^ \N)$  & \cmark                        & \cmark
                                          & 1                                          & ?                                       & \cite{main}                                     \\
    endomorphisms of \((\N,\leq)\)    & $\operatorname{End}(\N,\leq)$                      & \cmark    & ?
                                          &    $\geq\aleph_0$             & \cmark                            & \cref{thm:endN}                                    \\
endomorphisms of \((\N,<)\)    & $\operatorname{End}(\N,<)$                      &?    & \xmark
                                          & $2^{\aleph_0}$                                      & \cmark                        &\cref{thm:endNl}                                \\
unbounded endomorphisms of \((\N,\leq)\)    & $\operatorname{End}^{\infty}(\N,\leq)$                         & \cmark & ?
                                          & 1                                          & \cmark                             & \cref{thm:endNi}                                    \\
 endomorphisms of \((\Z,\leq)\)    & $\operatorname{End}(\Z,\leq)$                          & \cmark &?
                                          & $\geq \aleph_0$                   & \cmark                         &     \cref{thm:endZ}                                \\
    endomorphisms of \((\mathbb{Q},\leq)\)    & $\operatorname{End}(\mathbb{Q},\leq)$                       & \cmark   & \xmark
            & 1                                           & \cmark                        & \cite{PiSc}+\cite{main}                   \\
 endomorphisms of the random graph &$\operatorname{End}(R)$& \cmark   & \cmark
                                          & 1          & \cmark                     & \cite{EJMPP}+\cite{KR}+\cite{Hrushovski1992aa}                 \\
endomorphisms of the random poset&$\operatorname{End}(P)$& \cmark   & ?
                                          & 1          & \cmark                     & \cite{EJMPP}                  \\
\end{tabular}

\medskip

\begin{tabular}{rcl}
      MP & = & The semigroup admits a maximum Polish semigroup topology.\\ 
    AC & = & The semigroup admits a second-countable topology with which it has automatic \\
    &&continuity with respect the class of second-countable topological semigroups.\\%&& %That is, there is no second-countable semigroup topology on this semigroup with\\&& respect to which all homomorphisms from it to second-countable topological\\&& semigroups are continuous. \\
    \# PST  & = & Number of Polish semigroup topologies \\
    Z       & = & The semigroup Zariski topology has a countable basis. \\
       (\cmark) &=& The assertion is true. \\
       (\xmark) &=& The assertion is false.\\
    (?) &=& Unknown.
\end{tabular}

\caption{Summary of Main Results}
\label{table-the-only}
\end{figure}

\section{Maximal Polish semigroup topologies and property $\mathbb{XX}$}

%Mention that the supremum of two semigroup topologies is a semigroup topology.
In this section, we develop techniques of finding maximal Polish topologies on a given semigroup.
The following useful lemma follows from Lemma~13.2 and Lemma~13.3 of~\cite{Kechris1995}.

\begin{lemma}[Folklore]\label{PolishBuilding}
If \(\Tau_0\subseteq \Tau_1\subseteq \ldots\) is  a sequence of Polish topologies on a set \(X\), then the least topology containing \(\bigcup_{i\in \N}\Tau_i\) is Polish as well.
In particular, if \(\Tau\) is a Polish topology on a set \(X\) and \((K_i)_{i\in \N}\) is a sequence of sets which are closed with respect to \(\Tau\), then the least topology containing \(\Tau\) and all of the sets \(K_i\) is Polish.
\end{lemma}

\begin{definition}
    If \(\mathcal{T}\) is a topology on a set \(X\), then we denote the set of Borel subsets of \(X\) with respect to \(\mathcal{T}\) by \(\mathcal{B}(\mathcal{T})\).
\end{definition}
A function $f:X\rightarrow X$ is called {\em Borel measurable} if $(U)f^{-1}$ is Borel, for each open subset $U$ of $X$.
The following proposition follows from Corollary 15.2 in~\cite{Kechris1995}.

\begin{proposition}
  \label{prop-borel-measurable}
  If $X$ and $Y$ are Polish spaces and $f: X \to Y$ is a Borel measurable
  bijection, then $f ^ {-1}$ is Borel measurable also. In particular, if $\tau_1,\tau_2$ are Polish topologies on the same set such that $ \tau_1\subseteq \mathcal B(\tau_2)$, then $ \mathcal B(\tau_1)= \mathcal B(\tau_2)$.
\end{proposition}

\begin{definition}
If \(A\subseteq X\), then we denote the subspace topology that \(\mathcal{T}\) induces on \(A\) by \(\mathcal{T}|_A\).
\end{definition}

A subset $A$ of a space $X$ is called \emph{meager} if $A$ is contained in the union of countably many nowhere dense sets. A subset $A\subseteq X$ is called {\em comeager}, if the set $X\setminus A$ is meager. The following proposition is well-known but we include its short proof.
\begin{proposition}[Folklore]\label{comeager}
Suppose that \(\Tau_1, \Tau_2\) are topologies on a set \(X\) where \(\Tau_2\) is second-countable.
\begin{enumerate}[\rm(1)]
    \item If \(\mathcal{B}(\Tau_2)\subseteq \mathcal{B}(\Tau_1)\), then there is a subset \(C\) of \(X\) which is comeager with respect to \(\Tau_1\) such that \(\Tau_2|_C\subseteq\Tau_1|_C\).
    \item  If \(\Tau_1\subseteq \Tau_2\) are Polish,  then there is a subset \(C\) of \(X\) which is comeager with respect to \(\Tau_1\) such that \(\Tau_1|_C=\Tau_2|_C\).
\end{enumerate}
\end{proposition}
\begin{proof}
\((1):\) Let \(\makeset{U_i}{\(i\in \N\)}\) be a countable basis for \(\Tau_2\). For each \(i\in \mathbb{N}\), we have that \(U_i\in \mathcal{B}(\Tau_2)\subseteq \mathcal{B}(\Tau_1)\). Thus for all \(i\), the set \(U_i\) is Borel in \(\Tau_1\). Proposition 8.22 of \cite{Kechris1995} implies that for each $i\in\N$ there are a meager set \(M_i\) and open set \(U_i'\) (both with respect to \(\Tau_1\)) such that \(U_i\backslash M_i = U_i'\backslash M_i\). 
Let \(C =X\backslash (\bigcup_{i\in \N} M_i)\). By construction, the set \(C\) is comeager with respect to \(\Tau_1\), and all of the sets \(U_i\cap C\) belong to \(\Tau_1|_C\). 
As these sets are a basis for \(\Tau_2|_C\), the result follows.

\((2):\)
By \cref{prop-borel-measurable}, it follows that \(\mathcal{B}(\Tau_2)= \mathcal{B}(\Tau_1)\).
Thus by \((1)\), there is a subset \(C\) which is comeager with respect to \(\Tau_1\) with \(\Tau_2{\restriction}_C\subseteq\Tau_1{\restriction}_C\). As \(\Tau_1\subseteq \Tau_2\), the result follows.
\end{proof}

\begin{definition}
    We say that a topology \(\Tau\) on a set \(X\) is \emph{regular} if every neighborhood of each point in \(X\) contains a closed neighborhood of that point.
\end{definition}
Note that we don't assume regular spaces to be Hausdorff. In fact, \cref{regfromreg}(2) is false if we replace all instances of the word ``regular'' with the words ``regular Hausdorff''.
However we are primarily concerned with Polish spaces and we discuss regular spaces in order to arrive at \cref{borelcontinuous-XX}, which is applicable to metrizable spaces.
\begin{lemma}\label{regfromreg}
    Suppose that \(X, Y\) are sets, \(f:X\to Y\) is a function and \(\Tau_1, \Tau_2\) are regular topologies on \(Y\). The following topologies are regular:
    \begin{enumerate}[\rm(1)]
        \item The topology on \(Y\) generated by \(\Tau_1\) and \(\Tau_2\).
        \item The topology \((\Tau_1)f^{-1}:=\makeset{(U)f^{-1}}{\(U\in \Tau_1\)}\) on \(X\).
    \end{enumerate}
\end{lemma}
\begin{proof}
    \((1):\) Let \(y\in Y\). The set \[\makeset{N_1\cap N_2}{\(N_1\) is a \(\Tau_1\) neighborhood of \(y\), 
   \(N_2\) is a \(\Tau_2\) neighborhood of \(y\)}\] is a neighbourhood basis for \(y\) in the topology generated by \(\Tau_1\) and \(\Tau_2\).
   Let \(N_1\cap N_2\) be such a neighbourhood, we need only show that it contains a closed neighborhood of \(y\). As \(\Tau_1\) and \(\Tau_2\) are regular, there are closed neighborhoods \(N_1'\) and \(N_2'\) of \(y\) with respect to \(\Tau_1\) and \(\Tau_2\) respectively such that \(N_1'\subseteq N_1\) and \(N_2'\subseteq N_2\). Thus the neighborhood \(N_1'\cap N_2'\) is closed in the topology generated by \(\Tau_1\) and \(\Tau_2\). As \(N_1'\cap N_2'\subseteq N_1\cap N_2\) the result follows.

   \((2):\) Let \(x\in X\) and suppose \(x\in V\in (\Tau_1)f^{-1}\). It follows that there is \(U\in \Tau_1\) with \(V=(U)f^{-1}\) and hence \((x)f\in U\). As \(\Tau_1\) is regular, there is a closed neighborhood \(U'\) of \((x)f\) with \(U'\subseteq U\). Thus $(Y\setminus U')f^{-1}$ is an open set, witnessing that \((U')f^{-1}\subseteq V\) is a closed neighborhood of \(x\) as required.
\end{proof}

% \begin{lemma}\label{lem:dense_cont_points}
% Suppose that \((X, \Tau_1)\) is a metrizable space, \((X,\Tau_2)\) is a first countable space and \(D\subseteq X\) is dense with respect to the topology generated by \(\Tau_1\) and \(\Tau_2\).
% Suppose further that \(\Tau_1|_D\subseteq \Tau_2|_D\).
% It follows that the identity map \(id:(X, \Tau_2)\to (X, \Tau_1)\) is continuous at each point in \(D\).
% \end{lemma}
% \begin{proof}
% Let \(d_1\) be a metric compatible with \(\Tau_1\).
% Let \(a\in D\) be arbitrary and let \(U_0\supseteq U_1 \ldots\) be a neighbourhood basis for \(a\) with respect to \(\Tau_2\) consisting of open sets.
% Moreover suppose that \((a_i)_{i\in \N}\) is a sequence in \(X\) converging to \(a\) with respect to \(\Tau_2\) such that \(a_i\in U_i\). To conclude that \(id:(X, \Tau_2)\to (X, \Tau_1)\) is continuous at \(a\), we need only show that \((a_i)_{i\in \N}\) converges to \(a\) with respect to \(\Tau_1\).

% For each \(i\in \mathbb{N}\) let \(a_i'\in D\cap U_i\cap \makeset{b\in S}{\(d_1(b,a_1)<\frac{1}{2^i}\)}\) (this exists by the choice of \(D\)). Thus the sequence \((a_i')_{i\in \N}\) converges to \(a\) with respect to \(\Tau_2\).
% As \(\Tau_2|_A\supseteq \Tau_1|_A\), it follows that \((a_i')_{i\in \N}\) converges to \(a\) with respect to \(\Tau_1\) as well. 
% As for all \(i\in\N\) we have \(d_1(a_i,a_i')<\frac{1}{2^i}\) it follows that \((a_i)_{i\in \N}\) converges to \(a\) with respect to \(\Tau_1\) as well.

% \end{proof}

The following technical lemma is useful for finding the points of continuity of the identity function on a set $X$, when $X$ is endowed with multiple topologies.

\begin{lemma}\label{lem:dense_cont_points}
Suppose that \(X\) is a set and \(\Tau_1, \Tau_2\) are topologies on \(X\) with \(\Tau_2\) regular. Suppose further that \(D\subseteq X\) is dense with respect to the topology generated by \(\Tau_1\) and \(\Tau_2\), and the identity map \(\id:(D, \Tau_1|_D)\to (D, \Tau_2|_D)\) is continuous.
In this case the identity map \(\id:(X, \Tau_1)\to (X, \Tau_2)\) is continuous at each point in \(D\).
\end{lemma}
\begin{proof}
    Let \(d\in D\) and \(U\) be a neighborhood of \(d\) with respect to \(\Tau_2\), we only need to show that \(U\) is a neighborhood of \(d\) with respect to \(\Tau_1\).
    As \(\Tau_2\) is regular, we may assume without loss of generality that \(U\) is closed with respect to \(\Tau_2\).

    The set \(U\cap D\) is a closed neighborhood of \(d\) in the topological space \((D, \Tau_2|_D)\). Thus \(U\cap D\) is a closed neighborhood of \(d\) in the topological space \((D, \Tau_1|_D)\). Let \(V\in \Tau_1\) be such that \(d\in V\) and \(V\cap D\subseteq U\). It suffices to show that \(V\subseteq U\).
    In other words, we only need to show that \(V\backslash U=\varnothing\). As \(U\) is closed in \(\Tau_2\) and \(V\) is open in \(\Tau_1\), the set \(V\backslash U\) is open in the topology generated by \(\Tau_1\) and \(\Tau_2\). Thus \(D\cap (V\backslash U)\) is dense in \(V\backslash U\). But \(D\cap V\subseteq U\), so \(D\cap (V\backslash U)=\varnothing\) is dense in \(V\backslash U\). Thus \(V\backslash U = \varnothing\) as required. 
\end{proof}
% The following result follows from Proposition 8.22 from \cite{Kechris1995}.
% \begin{theorem}[Proposition 3.5.1]\label{almost continuous}
% Let \(T\) be a topological space. For all Borel sets \(B\) there exists an open set \(U_B\) such that \(U_B\triangle B\) is meager.
% \end{theorem}
A map $f: X\rightarrow Y$ is called {\em open} if $(U)f$ is open for each open subset $U$ of $X$. To keep the paper self-contained, we include a short proof of the following folklore fact.

\begin{lemma}[Folklore]\label{meager preimage}
The preimage of a meager set under an open continuous map is meager.
\end{lemma}
\begin{proof}
Let $X, Y$ be topological spaces and $f: X\rightarrow Y$ be an open continuous map. Fix any meager set $B\subseteq Y$. Then there exist nowhere dense subsets $A_n$, $n\in\w$ such that $B=\bigcup_{n\in\w}A_n$. Since $(B)f^{-1}=\bigcup_{n\in\w}(A_n)f^{-1}$, it suffices to show that $(A_n)f^{-1}$ is nowhere dense in $X$ for every $n\in\w$. 
Assuming the contrary, there exist $n\in\w$ and a nonempty open set $U$ such that $U\subseteq \overline{(A_n)f^{-1}}$. Observe that
$$(U)f\subseteq (\overline{(A_n)f^{-1}})f\subseteq \overline{(A_n)f^{-1}f}\subseteq\overline{A_n},$$
where the first inclusion holds by the choice of $U$, and the second inclusion follows from the continuity of $f$. 
Thus $(U)f$ is a non-empty open subset of $\overline{A_n}$, which contradicts the choice of $A_n$. Hence $(B)f^{-1}$ is meager. 
\end{proof}

The following lemma is the main tool for proving \cref{borelcontinuous-XX}. %The hypotheses are quite technical but they will be satisfied when $(S,\Tau_1)$ is a topological semigroup with property \(\mathbb{XX}\).

\begin{lemma}\label{lem:bigXX}
    Suppose that \(S\) is a non-empty set, \(\Tau_1\) is a Polish topology on \(S\), \(\Tau_2\) is a second-countable topology on \(S\), and  \(A\) is dense \(G_\delta\) in \((S, \Tau_1)\).
    Suppose further that for all \(a, b\in A\) there is a function \(t_{a, b}:S\to S\) such that
    \begin{itemize}
        \item \(t_{a, b}:(S,\Tau_1)\to (S,\Tau_1)\) is continuous;
        \item \(t_{a, b}:(S,\Tau_2)\to (S,\Tau_2)\) is continuous;
        \item \((a)t_{a,b}=b\);
        \item the restriction \(t_{a, b}{\restriction}_A:(A,\Tau_1|_A)\to (S,\Tau_1)\) is an open map.
    \end{itemize} 
    In this case the following hold:
    \begin{enumerate}
        \item Every comeager subset of \(A\) with respect to \(\Tau_1\) is dense in \(A\) with respect to \(\Tau_2|_A\).
    
        \item If \(\Tau_2\) is regular and \(\mathcal{B}(\Tau_2)\subseteq \mathcal{B}(\Tau_1)\), then \(\id:(S, \Tau_1)\to (S, \Tau_2)\) is continuous at every point in \(A\).
 \end{enumerate}

\end{lemma}
\begin{proof}
    \((1):\) Let \(C\) be a comeager subset of \(A\) with respect to \(\Tau_1\). It follows that \(S\backslash C\) is meager with respect to \(\Tau_1\).
    Suppose for a contradiction that there is \(U\in \Tau_2\) with \(U\subseteq S\backslash C\) but \(U\cap A\neq \varnothing\).
    Let \(\mathcal{B}\) be a countable basis for \(\Tau_2\). Let \(M:=\makeset{B\in \mathcal B}{\(B\) is meager 
    with respect to \((S,\Tau_1)\)}\). It follows that \(V:=\bigcup M\supseteq U\) intersects \(A\).
    Note that \(V\) contains every subset of \(S\) which is meager with respect to \(\Tau_1\) and open with respect to \(\Tau_2\). Also note that $V$ is meager itself with respect to $\Tau_1$. As \(A\) is comeager with respect to \(\Tau_1\), there is \(a\in A\backslash V\).
    Also it follows from above that there is some \(c\in V\cap A\).
    Then \(a\in (V)t_{a,c}^{-1}\cap A\). As $t_{a,c}$ is continuous with respect to \(\Tau_2\), we have \((V)t_{a,c}^{-1}\cap A\in \Tau_2|_A\).

Observe that the preimage of $V$ under the map $t_{a,c}{\restriction}_A$ is equal to $(V)t_{a,c}^{-1}\cap A$. Since the set $V$ is meager with respect to $\Tau_1$ and the map $t_{a,c}{\restriction}_A$ is continuous and open, Lemma~\ref{meager preimage} implies that \((V)t_{a,c}^{-1}\cap A\) is meager in \((A, \Tau_1|_A)\).  
    Recall that \(a\in (V)t_{a,c}^{-1}\). The set \((V)t_{a,c}^{-1}\) is open in \((S,\Tau_2)\) and \[(V)t_{a,c}^{-1}\subseteq ((V)t_{a,c}^{-1}\cap A)\cup S\backslash A.\] So $(V)t_{a,c}^{-1}$ is meager in \((S,\Tau_1)\). It follows from the choice of \(V\) that \((V)t_{a,c}^{-1}\subseteq V\) and so \(a\in V\), a contradiction.

    (2):     As \(A\) is \(G_\delta\) with respect to \(\Tau_1\), the topology \(\Tau_1|_A\) is Polish.
    %\(A\) is \(G_\delta\) wrt \(\Tau_2\). Thus \((A, \Tau_2)\) is Polish.
 % {\color{blue} !Add reference!}  As $\Tau_1$ and $\Tau_2$ are comparable Polish topologies, the subspaces topologies on \(A\) wrt \(\Tau_1\) and \(\Tau_2\) have the same Borel structure.
 Let \(\Tau_2'\) be the topology generated by \(\Tau_1\) and \(\Tau_2\). By \cref{regfromreg}, the topology \(\Tau_2'\) is second-countable and regular. It is easy to check that
 the functions \(t_{a,b}:(S,\Tau_2')\to (S,\Tau_2')\) are all continuous.
Note also that since \(\mathcal{B}(\Tau_2)\subseteq \mathcal{B}(\Tau_1)\), we have
 \(\mathcal{B}(\Tau_2')= \mathcal{B}(\Tau_1)\).
 By Proposition~\ref{comeager}(1) applied to \((A,\Tau_1|_A)\) and \((A,\Tau_2'|_A)\), there is a comeager with respect to \(\Tau_1\) subset \(C\) of \(A\) such that \( \Tau_2'|_C \subseteq \Tau_1|_C\). 
    Thus by (1), the set \(C\) is dense in \((A, \Tau_2'|_A)\). 
    By Lemma~\ref{lem:dense_cont_points} the map \(\id:(A, \Tau_1|_A)\to (A, \Tau_2'|_A)\) is continuous at every point in \(C\).

    Let \(p\in C\) and \(q\in A\). Fix an open neighborhood \(U\) of \(q\) in \((S, \Tau_2)\). Then
    \((U)t_{p, q}^{-1}\cap A\) is an open neighborhood of \(p\) in \((A, \Tau_2|_A)\). Since $p\in C$ and \(\id:(A, \Tau_1|_A)\to (A, \Tau_2'|_A)\) is continuous at every point in \(C\), \((U)t_{p, q}^{-1}\cap A\) is an open neighborhood of \(p\) in \((A, \Tau_1|_A)\). Thus \(((U)t_{p, q}^{-1}\cap A)t_{p, q}\subseteq U\) is a neighbourhood of \(q\) with respect to \((S, \Tau_1)\).
    We have now shown that an arbitrary open neighbourhood \(U\) of \(q\) with respect to \(\Tau_2\) is a neighborhood of \(q\) with respect to \(\Tau_1\) as required.
\end{proof}

%Next we introduce a property that will be central to this paper.

%Following~\cite{group}, a function $f: S\rightarrow S$ of the form $(x)f=a_0xa_1\cdots xa_n$ is called a {\em semigroup polynomial of degree $n$}. We agree that constant functions are semigroup polynomials of degree $0$.
 
The following definition is a partial case of \cite[Definition 7.5]{LunaPhD}. It is also  reminiscent of property \(\overline{\mathbf{X}}\) used in \cite{PiSc}.

\begin{definition}
    We say that a topological semigroup \((S,\Tau)\) has property \(\mathbb W\) with respect to a subset \(A\subseteq S\) if for all \(s\in S\) there exist $a\in A$ and $\{b_0,\ldots,b_n\}\subseteq S^1$ such that the function $\phi: S\rightarrow S$ defined by $(x)\phi=b_0xb_1\cdots xb_n$ has the following properties:
    \begin{enumerate}[\rm(1)]
    \item $(a)\phi=s$;
    \item if \(a\in U\in \Tau|_A\), then \(s\) belongs to the interior of \( (U)\phi\) with respect to \(\Tau\).
    \end{enumerate}
\end{definition}

Note that property \(\mathbb W\) is implied by property $\mathbb{XX}$ (recall \cref{XX}) when \(S\) is a topological semigroup. The following lemma follows from \cite[Lemma 7.6]{LunaPhD}.

\begin{lemma}\label{propertyW}
    If \(S\) is a topological semigroup which has property \(\mathbb W\) with respect to \(A\subseteq S\) and \(\phi:S\to T\) is a semigroup homomorphism of topological semigroups such that \(\phi{\restriction}_A\) is continuous, then \(\phi\) is continuous.
\end{lemma}

The following three interconnected theorems are useful to find maximal Polish semigroup topologies on a given semigroup.

%Note that if a topological semigroup $S$ has property $\mathbb{XX}$ with respect to $A\subseteq S$, then $S$ has property $\mathbb{W}$ with respect to $A$. %So, Lemma~\ref{lem:bigXX}(2) implies the following.

\begin{theorem}\label{cor:XX_sub}
Let \(\Tau_1\) and \(\Tau_2\) be semigroup topologies on a semigroup \(S\) such that 
\begin{itemize}
    \item \(\Tau_1\) is Polish;
    \item \(\Tau_2\) is regular and second-countable;
    \item \(\mathcal{B}(\Tau_2)\subseteq \mathcal{B}(\Tau_1)\);
    \item \((S, \Tau_1)\) has property $\mathbb{XX}$ with respect to a subset \(A\) of \(S\).
\end{itemize}
 Then \(\Tau_2\subseteq\Tau_1\).
\end{theorem}
\begin{proof}
Since $S$ is a topological semigroup, each of the maps \(t_{a,s}\) from  property $\mathbb{XX}$ is continuous.
Taking into account that $(S,\Tau_1)$ has property $\mathbb{XX}$, Lemma~\ref{lem:bigXX}(2) implies that the identity map \(\id:(S, \Tau_1)\to (S, \Tau_2)\) is continuous at every point in \(A\).
As property $\mathbb{XX}$ implies property $\mathbb{W}$, it follows from \cref{propertyW} that \(\id:(S, \Tau_1)\to (S, \Tau_2)\) is continuous. Hence \(\Tau_2\subseteq\Tau_1\).
\end{proof}

\begin{theorem}\label{borelcontinuous-XX}
    Let \(S\) and \(T\) be topological semigroups such that
    \begin{itemize}
        \item \(S\) is Polish;
        \item \(T\) is regular and second-countable;
        \item  \(S\) has property $\mathbb{XX}$ with respect to some \(A\subseteq S\).
   \end{itemize}
    Then every Borel measurable homomorphism from \(S\) to \(T\) is continuous. 
\end{theorem}
\begin{proof}
Let \(\phi:S\to T\) be a Borel measurable homomorphism.  Let \(\Tau_1\) be the topology on \(S\) and \(\Tau_2\) be the topology \(\makeset{(U)\phi^{-1}}{\(U\subseteq T\) is open}\). The topology \(\Tau_2\) is a second-countable semigroup topology on \(S\) and is also regular by \cref{regfromreg}. It follows from \cref{cor:XX_sub} that \(\Tau_2\subseteq \Tau_1\).
    Thus \(\phi\) is continuous.
\end{proof}

\begin{theorem}\label{cor:XX_main}
If \((S,\Tau)\) is a Polish topological semigroup that has property $\mathbb{XX}$ with respect to a subset \(A\) of \(S\), then $\Tau$ is a maximal Polish semigroup topology on $S$.
\end{theorem}
\begin{proof}
    This follows from \cref{borelcontinuous-XX} and \cref{prop-borel-measurable}.
\end{proof}

\section{$\E(\N,<)$}
%The Zariski topology on \(\operatorname{End}(\N, <)\) can be found using standard techniques.

The goal of this section is proving each item of
\cref{thm:endNl}.
The following useful result was proven in~\cite[Lemma 5.3]{main}.
\begin{lemma}\label{lem-zariski-pointwise}
  Let \(X\) be an infinite set and let \(S\) be a subsemigroup of \(X^X\) such
  that for every \(x \in X\) there exist \( a, b, c_{0}, \ldots, c_{n-1} \in S \)
  for some $n\in \N$ such that the following hold:
  \begin{enumerate}
    \item\label{lem-zariski-pointwise-i} 
    \((y)a = (y)b\) if and only if \(y \neq x\);

    \item\label{lem-zariski-pointwise-ii} 
          \(x \in \im(c_{i})\) for all \(i\);

    \item\label{lem-zariski-pointwise-iii} 
          for every \(s \in S\) and every \(y \in X\setminus \{(x)s\}\) there
          is \(i \in \{0, \ldots, n-1\}\) so that \( \im(c_{i}) \cap (y)s^{-1} =
          \varnothing\).
  \end{enumerate}
  Then the semigroup Zariski topology of \(S\) is the pointwise topology.
\end{lemma}

\begin{lemma}\label{lem:zariski_pointwise}
    The semigroup Zariski topology on \(\operatorname{End}(\N, <)\) is the pointwise topology.
\end{lemma}
\begin{proof}
    We need only find for all \(x\in \N\), some \(a, b, c_0, c_1\in \operatorname{End}(\N, <)\) as in Lemma~\ref{lem-zariski-pointwise}. Fix $x\in \N$ and let \(a, b, c_0,c_1:\N\to \N\) be defined by
     \[(y)a=\begin{cases}
         y & y\leq x;\\
         y+2 & y>x.
     \end{cases}\quad\quad\quad(y)b=\begin{cases}
         y & y< x;\\
         x+1 & y=x;\\
         y+2 & y>x.
     \end{cases}\]

       %  Let \(c_0, c_1:\N\to \N\) be defined by
     \[(y)c_0=2y+x.\quad\quad\quad\quad(y)c_1=\begin{cases}
            2y+1+x & y\neq 0;\\
            x & y=0.
        \end{cases}\]
We check that the conditions from Lemma~\ref{lem-zariski-pointwise} are satisfied. For condition (1), we have \((x)a=x\neq x+1=(x)b\) and for \(y\neq x\), we have \((y)a=(y)b\).
For condition (2), we have \((0)c_0=(0)c_1=x\).
For condition (3), fix any \(s\in \operatorname{End}(\N,<)\). Note that \(s\) is injective. If \(y\in \N\backslash \{(x)s\}\) then either \(y\notin \im(s)\) or \((y)s^{-1}=((z)s)s^{-1}=z \) for some \(z\neq x\). Thus, as \(\im(c_0)\cap\im(c_1)=\{x\}\), either \((y)s^{-1}\cap\im(c_0)=\varnothing\) or \((y)s^{-1}\cap \im(c_1)=\varnothing\).     
\end{proof}

We now prove two useful lemmas for defining large families of topologies on $\operatorname{End}(\N, <)$ in order to establish the remaining parts of \cref{thm:endNl}.

\begin{definition}
    An ideal $J$ of a semigroup $S$ is called {\em prime} if $S\setminus J$ is a subsemigroup of \(S\). 
\end{definition}

\begin{lemma}\label{sem}
Let $(S,\tau)$ be a topological semigroup and $J\subseteq S$ be a prime ideal.  Let $\tau(J)$ be the topology on $S$ generated by the subbase $\tau\cup\{J\}$. Then $(S,\tau(J))$ is a topological semigroup.   
\end{lemma}

\begin{proof}
Consider the subsemigroup \(B:=\{0, 1\}\) of the multiplicative semigroup of the integers. It is clear that \(\Tau_B:=\{\varnothing, \{0\}, \{0,1\}\}\) is a semigroup topology on $B$. 
As \(J\) is a prime ideal, the map \(\phi:S\to B\) defined by \((s)\phi=0\iff s\in J\) is a semigroup homomorphism.
Since the topology \(\tau(J)\) is generated by the semigroup topologies \(\tau\) and \((\Tau_B)\phi^{-1}\) on \(S\), it is also a semigroup topology.
\end{proof}

\begin{lemma}\label{Pol2}
Let $J$ be an ideal on a Polish topological semigroup $(S,\tau)$ such that the set $S\setminus J$ is countable. Then the topology $\tau^*(J)$ generated by $\tau\cup\{\{x\}:x\in X\setminus J\}\cup\{J\}$ is a Polish semigroup topology on $S$.   
\end{lemma}

\begin{proof}
Let us show that $\tau^*(J)$ is a semigroup topology on $S$. 
Let \(Q\) be the Rees quotient of \(S\) by \(J\) equipped with the discrete topology \(\Tau_Q\) and let \(\phi:S\to Q\) be the quotient map. 
The topology $\tau^*(J)$ is generated by the semigroup topologies \(\tau\) and \((\Tau_Q)\phi^{-1}\) on \(S\) and is thus a semigroup topology.
 
\cref{PolishBuilding} implies that the topology  $\tau_J$ generated by $\tau$ and all the singletons $f\in S\setminus J$ is a Polish topology. Note that $J$ is closed in $(S,\tau_J)$.  Hence applying \cref{PolishBuilding} again we get that the topology $\tau^*(J)$ is Polish.
\end{proof}

For any functions $f,g\in \N^\N$ set
\begin{itemize}
   % \item $f<g$ if $(n)f<(n)g$ for all $n\in\N$;
    \item $f\leq g$ if $(n)f\leq (n)g$ for all $n\in\N$;
    \item $f\leq ^* g$ if the set $\{n\in\N: (n)f>(n)g\}$ is finite.
   % \item $f>g$ if $(n)f>(n)g$ for all $n\in\N$;
  % \item $f\geq g$ if $(n)f\geq (n)g$ for all $n\in\N$;
   %\item $f\geq ^* g$ if the set $\{n\in\N: (n)f<(n)g\}$ is finite.
\end{itemize}

For an element $f\in \End$ set
\begin{itemize}
\item ${\downarrow}f=\{g\in \End: g\leq f\}$;
\item ${\downarrow}^*f=\{g\in \End: g\leq^* f\}$.
\end{itemize}
For a subset $A\subseteq \N^\N$ let ${\downarrow}^*A=\bigcup_{a\in A}{\downarrow}^*a$.
% For any subset $T$ of a subsemigroup $S\subseteq \N^\N$ let 
% $${\downarrow}^{<}T=\bigcup_{x\in T}{\downarrow}^{<}x,
% \qquad {\downarrow}^{\leq}T=\bigcup_{x\in T}{\downarrow}^{\leq}x, \qquad {\downarrow}^*T=\bigcup_{x\in T}{\downarrow}^*x,$$
% $${\uparrow}^{<}T=\bigcup_{x\in T}{\uparrow}^{<}x,
% \qquad {\uparrow}^{\leq}T=\bigcup_{x\in T}{\uparrow}^{\leq}x, \qquad {\uparrow}^*T=\bigcup_{x\in T}{\uparrow}^*x.$$

 A family $\Psi\subseteq \N^\N$ is called {\em bounded} if there exists $f\in\N^\N$ such that $g\leq^* f$ for any $g\in \Psi$. A family $\Psi\subseteq \N^\N$ is called {\em unbounded} if $\Psi$ is not bounded. Let
 $\mathfrak b$ be the minimal cardinality of an unbounded family in $(\N^\N,\leq^*)$.
 It is well-known that $\aleph_1\leq \mathfrak b\leq \mathfrak c=2^{\aleph_0}$. 

%We now show part (2) of \cref{thm:endNl}.
\begin{lemma}\label{prop:big_ideal_chain}
There exists a family $\{J_{\alpha}:\alpha\in\mathfrak b\}$ of prime ideals in $\E(\N,<)$ such that \(J_\alpha \subsetneqq J_\beta\) whenever \(\alpha>\beta\).  
\end{lemma}

\begin{proof}
Fix any $f_0\in \End$. Let 
$J_0=\End\setminus {\downarrow}^* \{f_0^{n}:n\in\N\setminus\{0\}\}$. %Observe that for each $h,g\in\End$ we have  $h\leq hg$ and $h\leq gh$. 
We claim that $J_0$ is a prime ideal. Firstly, \(J_0\) is an ideal as it is upwards closed. If \(g\leq^* f_0^n\) and \(h\leq^* f_0^m\), then \(gh\leq^* f_0^nf_0^m=f_0^{n+m}\). Thus \(\operatorname{End}(\N,<)\backslash J_0\) is a semigroup.

Assume that for some $\xi\in\mathfrak b$ we already constructed a decreasing chain $\{J_{\alpha}:\alpha\in\xi\}$ of prime ideals and a family $\{f_\alpha:\alpha\in\xi\}\subseteq \End$ such that $\End\setminus J_{\alpha}={\downarrow}^*\{f_\alpha^n:n\in\N\setminus\{0\}\}$. Since $\xi<\mathfrak b$ and $\mathfrak b$ is uncountable, we get that 
 $$|\{f^n_{\alpha}:n\in\N\setminus\{0\}, \alpha\in\xi\}|=\max\{|\xi|,\aleph_0\}<\mathfrak b.$$   Then there exists $g\in\N^\N$ such that $f_{\alpha}^n\leq^* g$ for each $n\in\N\setminus\{0\}$ and $\alpha\in\xi$. Define a function $f_{\xi}\in \End$ as follows: $(n)f_\xi=\sum_{i\leq n}(i)g+1$.  It is clear that 
 $\{f^n_{\alpha}:n\in\N\setminus\{0\}, \alpha\in\xi\}\subseteq {\downarrow}^*f_{\xi}$. Let $J_{\xi}=\End\setminus {\downarrow}^*\{f_{\xi}^n:n\in\N\setminus\{0\}\}$. It is easy to see that $J_{\xi}$ is a prime ideal, which is contained in $\bigcap_{\alpha\in\xi}J_{\alpha}$.
\end{proof}

\begin{proposition}\label{nosecondAC}
  There is no maximal second-countable semigroup topology on $\End$. 
\end{proposition}
\begin{proof}
    Suppose for a contradiction that \(\Tau\) is a maximal second-countable semigroup topology on \(\End\).
    Let \(\{J_{\alpha}:\alpha\in\mathfrak b\}\) be the chain from Proposition~\ref{prop:big_ideal_chain} and note that \(J_{\alpha}\subseteq J_{\beta} \iff \beta \leq \alpha\). 
    Lemma~\ref{sem} implies that \(\Tau (J_\alpha)\) is a second-countable semigroup topology on \(\End\) for all \(\alpha \in \mathfrak b \). 
    As \(\Tau\) is maximal, it follows that \(\Tau (J_\alpha)=\Tau\). Thus for all \(\alpha\in \mathfrak{b}\) we have \(J_\alpha\in \Tau\).

    Let \(B\) be a countable basis for \(\Tau\). For every \(\alpha\in \mathfrak b\), fix any \(f_\alpha \in J_\alpha\setminus J_{\alpha+1}\). For all \(\alpha\in \mathfrak b\), there exists \(U_\alpha\in B\) such that \(f_\alpha\in U_\alpha\subseteq J_\alpha\in \Tau\). 
    Hence for all \(\alpha\in \mathfrak b\) we have \(U_\alpha\subseteq J_\alpha\) but \(U_\alpha\subsetneq J_{\alpha+1}\). It follows that the map from \(\mathfrak b\) to \(B\) defined by \(\alpha\mapsto U_\alpha\) is injective, a contradiction.
\end{proof}

\begin{corollary}
 The semigroup $\End$ admits no second-countable semigroup topology which has automatic continuity with respect to the class of second-countable topological semigroups. 
 % To Do: Check maybe it doesn't admit a second-countable topology such that each Borel measurable homomorphism from it to a second-countable Hausdorff topological semigroup is continuous. For this, it suffices to show that all $J_\alpha$ are Borel with respect to pointwise topology. Even weaker is enough; $J_{\alpha+1}$ must be Borel in $\tau_\alpha$. we can take $J_\alpha$ to be countable, and the length of induction may be $\w_1$. So, everything should work.
\end{corollary}

We finish the section by showing items (3) and (4) of \cref{thm:endNl}.

% \begin{lemma}
% Let $S$ be a cancellative  semigroup and a point $c\in S$ be isolated. Then for any $a,b\in S$ such that $ab=c$, the points $a$ and $b$ are isolated.    
% \end{lemma}

% \begin{proof}
% To derive a contradiction, assume that there exists $a,b\in S$ such that $a$ is not isolated and $c=ab$ is isolated. Since $S$ is a  semigroup, there exists an open neighborhood $U$ of $a$ such that $Ub\subseteq \{c\}$. Since $S$ is cancellative and $|U|\geq 2$, we get that $|Ub|\geq 2$, which contradicts the above inclusion.     
% \end{proof}

\begin{definition}
 We define 
 \[\operatorname{AS}(\N,<):=\makeset{f\in \operatorname{End}(\N, <)}{\(\N\backslash \im(f)\) is finite}. \]
 Note that \(\operatorname{AS}(\N,<)\) is a countable semigroup and \(\operatorname{End}(\N, <)\backslash \operatorname{AS}(\N,<)\) is an ideal of \(\operatorname{End}(\N, <)\).
 We refer to the elements of \(\operatorname{AS}(\N,<)\) as \emph{almost surjections}.
\end{definition}

In the following proof it is convenient for us to identify natural numbers with finite ordinals, i.e., $n=\{0,\ldots, n-1\}$ for all $n\in\N$.

\begin{theorem}\label{continuun many things}
    There are precisely \(2^{\aleph_0}\) Polish semigroup topologies on \(\End\) and at least \(2^{\aleph_0}\) second-countable metrizable semigroup topologies on \(\End\) which are not Polish.
\end{theorem}
\begin{proof}
     For all \(f\in \End\), let 
    \[C_f:= \operatorname{AS}(\N,<)\cap {\downarrow} f.\]
    Observe that for each $f\in\End$, $C_f$ is a countable coideal.
    We claim  that $f=g\iff C_f=C_g$ for all \(f, g \in \End\). Indeed, the necessity is obvious. Assume that $f,g\in \End$ are distinct. Then there exists $n\in\N$ such that $(n)f\neq (n)g$. 
    We may assume without loss of generality that $(n)f>(n)g$. 
    Define an order endomorphism $\pi$ as follows: $\pi{\restriction}_{n+1}=f{\restriction}_{n+1}$ and $(n+k)\pi=\pi(n)+k$ for all $k\in\N$. It is easy to see that $\pi\in C_f\setminus C_g$ and thus $C_f\neq C_g$.
    %(as \(C\)  is dense in \(\End\) with the pointwise topology).

    Hence the family \(\makeset{C_f}{\(f\in \End\)}\) has cardinality \(2^{\aleph_0}\). Let \(\tau\) be the pointwise topology on \(\End\).  By \cref{Pol2}, for each of these sets \(C_f\), we can define a new Polish semigroup topology \(\tau^*(\End\backslash C_f)\) on \(\End\) with respect to which \(C_f\) is the set of isolated points.
    Hence there are at least \(2^{\aleph_0}\) such topologies.

    Conversely, as shown in Lemma~\ref{lem:zariski_pointwise}, every Polish semigroup topology on \(\End\) contains the pointwise topology. 
    It follows from Proposition~\ref{prop-borel-measurable} that every Polish topology on \(\End\) has the same Borel sets as the pointwise topology.

    Let \(B\) be the set of Borel subsets of \(\End\) with respect to the pointwise topology. 
    Note that \(|B|=\mathfrak{c}\) and hence \(|[B]^{\aleph_0}|=\mathfrak{c}\) also (where \([B]^{\aleph_0}\) is the set of countably infinite subsets of \(B\)). Each Polish topology on \(\End\) has a basis in \([B]^{\aleph_0}\) and is completely determined by this basis. 
    Thus the set of Polish topologies on \(\End\) has at most \(|[B]^{\aleph_0}|=\mathfrak{c}\) elements, as required.

    We now switch to non-Polish semigroup topologies on $\End$. Let \(g\in  \End\) be fixed, not almost surjective, and such that \((0)g> 0\).
    The set $\End\backslash {\downarrow}g$ is an ideal of \(\End\). Note that the set $I:=\End\backslash (\bigcup_{n\in \N}{\downarrow}g^n)$ is also an ideal. Since $\bigcup_{n\in \N}{\downarrow}g^n$ is a subsemigroup of $\End$, we get that \(I\) is a prime ideal. 
    We show that for all \(f\in \End\) the sets \(I\backslash C_f\) and \(\End\backslash (I\cup C_f)\) are dense in \(\End \setminus C_f\) with respect to the pointwise topology, which in turn coincides with the topology $\tau^*(\End\backslash C_f)|_{\End\backslash C_f}$.
    For each $n\in\N$ and finite increasing map $\psi: n\rightarrow \N$, we can find an extension \(h\in \End\backslash C_f\) of \(\psi\) with \((m)h> (m)g^m\) for all \(m>n \). Hence the set $I\backslash C_f$ is dense in \(\End \setminus C_f\). 
    Since \((0)g>0\), for each $n\in\N$ and finite increasing map $\psi: n\rightarrow \N$, we can find an extension \(h\in \End\) of \(\psi\) with \((m)h< (m)g^{1+(n-1)\psi}\) for all \(m\in \N \). As \(g\) is not almost surjective, then so is $g^{1+(n-1)\psi}$.  Thus we can also choose \(h\) to be not almost surjective (in this case \(h\notin C_f\)). 
    We have now shown that \(\End \backslash (I\cup C_f)\) is dense in \(\Endd\backslash C_f\). 
    
    Observe that neither \(I\backslash C_f\) nor \(\End \backslash (I\cup C_f)\) contains isolated points as if \(f\) belongs to either of these sets, then the sequence \((f_n)_{n\in \N}\) defined by
    \[(x)f_n=\begin{cases}
        (x)f & x\leq n;\\
        (x)f+1 & x>n.
    \end{cases}\]
    is contained in the same set and converges to \(f\).
    
    As \(I\) is a prime ideal, the set \(\{\varnothing, I, \End \backslash I, \End\}\) is a semigroup topology on \(\End\). Thus for all \(f\in \End\), the topology \(\tau(f)\) generated by $$\tau^*(\End\backslash C_f)\qquad \hbox{ and } \qquad\{\varnothing, I, \End \backslash I, \End\}$$ is a second-countable semigroup topology.
    The topological space \((\End,\tau(f))\) is metrizable as well, as it is the topological sum of two metrizable spaces $I$ and $\End \backslash I$, where each of those carries the subspace topology inherited from \(\tau^*(\End\backslash C_f)\). Since neither the set \(I\backslash C_f\) nor the set \(\End \backslash (I\cup C_f)\) contains isolated points, the only isolated points of \((\End,\tau(f))\) are those of \(C_f\). Thus these topologies are all distinct.
    
    We claim that none of these topologies are Polish.
    Suppose for a contradiction that \(\tau(f)\) is Polish. As the set \(C_f\) is clopen, it follows that the subspace \(\End \backslash C_f\) is Polish also.
    Similarly, both of the clopen subspaces \(I\backslash C_f\) and \(\End \backslash (I\cup C_f) \) are Polish.
    However, these spaces carry the subspace topology from \(\tau^*(\End \backslash C_f)\). 
    Therefore, both these spaces are Polish subspaces of \((\End,\tau^*(\End \backslash C_f))\). Thus both  \(I\setminus C_f\) and \(\End \backslash (I\cup C_f)\) are \(G_\delta\) with respect to \(\tau^*(\End \backslash C_f)\). 
    It follows that the Polish space \(\End\backslash C_f\) can be partitioned into two dense \(G_\delta\) subsets, this contradicts the Baire Category Theorem.
\end{proof}

\section{$\operatorname{End}^\infty(\N,\leq)$}
The goal of this section is to prove each item of \cref{thm:endNi}. In doing so it will be convenient to state some of the results in sufficient generality to also apply them in the proof of \cref{thm:endZ}.

Recall that $\operatorname{End}^\infty(\N,\leq)$ is the submonoid of $\Endd$ consisting of all maps with infinite image. For an element $x$ of semigroup $S$ by $l_x$ we denote the left shift by $x$ in $S$, i.e. $l_x: S\rightarrow S$, $y\mapsto xy$. Similarly,  $r_x: S\rightarrow S$, $y\mapsto yx$ is the right shift by $x$ in $S$. A topology on a semigroup $S$ is called {\em shift-continuous} if all left and all right shifts in $S$ are continuous.

\begin{theorem}\label{magenta}
    The semigroup Zariski topology on \(\operatorname{End}^\infty(\N,\leq)\) is the pointwise topology.
\end{theorem}
\begin{proof}
Since the pointwise topology on $\operatorname{End}^\infty(\N,\leq)$ is a Hausdorff semigroup topology, we get that the semigroup Zariski topology on $\operatorname{End}^\infty(\N,\leq)$ is contained in the pointwise topology. To show the converse inclusion, it suffices to check that for any positive integer \(n\) and \(h\in \operatorname{End}^\infty(\N,\leq)\) the set
\[U_n^h:=\makeset{k\in \operatorname{End}^\infty(\N,\leq)}{\(k{\restriction}_n=h{\restriction}_n\)}\]
is a (not necessarily open) neighborhood of \(h\) with respect to the semigroup Zariski topology.

We first show that \(U_{n+2}^{1_\N}\) is open.
For all \(i\in \mathbb{N}\), let \(\iota_i\in \Endd\) be the unique injection with image \(\N\backslash \{i\}\). In particular, the maps \(\iota_i, \iota_{i+1}\) agree on all values except \((i)\iota_i=i+1\) and \((i)\iota_{i+1}=i\).

    For all \(i\in \N\) we define
    \[K_{i}:=\makeset{f\in \operatorname{End}^\infty(\N,\leq)}{\((i)f=(i+1)f\)}=\makeset{f\in \operatorname{End}^\infty(\N,\leq)}{$\iota_i f = \iota_{i+1} f$}.\]
     Similarly, for all \(i\in \N\) define
    \[I_{i}:=\makeset{f\in \operatorname{End}^\infty(\N,\leq)}{\(i\not \in \im(f)\)} =\makeset{f\in \operatorname{End}^\infty(\N,\leq)}{\(f \iota_i= f \iota_{i+1}\)}.\]
    Observe that for each $i\in\N$ the sets $K_i$ and $I_i$ are closed in the semigroup Zariski topology on $\operatorname{End}^\infty(\N,\leq)$.
      It is easy to check that  for all \(m\in \N\)\ \[U_{m}^{1_\N}=\left(\bigcap_{i<m-1}\operatorname{End}^\infty(\N,\leq)\backslash K_i\right) \cap\left( \bigcap_{i<m} \operatorname{End}^\infty(\N,\leq)\backslash I_i\right).\] 
      It follows that for each $m\in\N$ the set $U_{m}^{1_\N}$
    is open in the semigroup Zariski topology on $\operatorname{End}^\infty(\N,\leq)$.

Fix the increasing enumeration $\im(h)=\{x_n:n\in\N\}$. Define the map $f_{\min}\in \operatorname{End}^\infty(\N,\leq)$ by
$(n)f_{\min}=\min ((x_n)h^{-1})$. Let $h^{\im}$ be the function that sends $n$ to $x_n$ and note that \(h^{\im}=f_{\min}h\). 
Let $$(n)g=\begin{cases}
0,& \hbox{ if } 0\leq n\leq (0)h;\\
k,& \hbox{ if } (k-1)h<n\leq (k)h.
\end{cases}
$$

%Let \(g\) be the largest surjection such that \(h^{\text{image}}g = 1\). 
Since $\mathfrak Z$ is shift-continuous and $h^{\im}g=1_\N$, we get that 
\((U_{n+2}^1) r_g^{-1}\) 
is open in the semigroup Zariski topology and contains \(h^{\im}\), where $r_g$ is the right shift by $g$ in the monoid $\operatorname{End}^\infty(\N,\leq)$. Moreover, the set $$T:=(U_{n+2}^{1_\N}) r_g^{-1}\cap \bigcap_{i<{n+2}} (\operatorname{End}^\infty(\N,\leq)\backslash I_{x_i})$$ is open in $(\operatorname{End}^\infty(\N,\leq),\mathfrak Z)$ and contains $h^{\im}$. 
We show that $T\subseteq U_{n+1}^{h^{\im}}$, which would imply that $ U_{n+1}^{h^{\im}}$ is a neighborhood of $h^{\im}$ in $(\operatorname{End}^\infty(\N,\leq),\mathfrak Z)$.
 Suppose that \(\phi\in T\), i.e., \(\phi g \in U_{n+2}^{1_\N}\) and the image of \(\phi\) contains the first \(n+1\) points in the image of \(h\). Then for \(i<n+1\), \((i)\phi\) must belong to \((i)g^{-1}\). This implies that for \(i<n+1\), \((i)\phi\) is the only value in the image of \(\phi\) which belongs to \((i)g^{-1}\). In particular \((i)\phi=x_i\) for all $i<n+1$, as $x_i\in (i)g^{-1}$ for all $i<n+1$ by the definition of $g$ and by the choice of $x_i$. Hence $\phi\in U_{n+1}^{h^{\im}}$, as required.

%%%%%%%%%%%%%%%%%%%%%%%%%%%%%%%%%%%%%%%%%%%%%%%%%%%%%%%%%%%%%%%%%%% 
%We now have 
%\[h^{\im}\in\makeset{k\in \operatorname{End}^\infty(\N,\leq)}{for all \(i<n+2\) we have \(kg \iota_i\neq kg \iota_{i+1}\),\\ \(\iota_i kg \neq \iota_{i+1} kg\), and \(k\iota_{(i)h^{\text{image}}}\neq k\iota_{(i)h^{\text{image}}+1}\)} \subseteq U_{n+1}^{h^{\text{image}}}.\]

%%%%%%%%%%%%%%%%%%%%%%%%%%%%%%%%%%%%%%%%%%%%%%%%%%%%%%%%%%%%%%%%%%%%
We finally move to \(U_h^n\). Define $f_{\max}\in \operatorname{End}^\infty(\N,\leq)$ by $(n)f_{\max}=\max((x_n)h^{-1})$.  Let $$W:=(U_{n+1}^{h^{\im}})l_{f_{\min}}^{-1}\cap (U_{n+1}^{h^{\im}}) l_{f_{\max}}^{-1}.$$ Since \(f_{\max}h=f_{\min}h=h^{\im}\) and $U_{n+1}^{h^{\im}}$ is a neighborhood of $h^{\im}$, we get that $W$ is a neighborhood of $h$ in $(\operatorname{End}^\infty(\N,\leq),\mathfrak Z)$. It remains to show that $W\subseteq U_{n}^h$.

%\[h\in ((U_{n+1}^{h^{\text{image}}})l_{f_l}^{-1}\cap (U_{n+1}^{h^{\text{image}}}) l_{f_u}^{-1})\subseteq U_{n}^h.\]
Suppose that \(\phi\in W\), i.e. both $f_{\min}\phi$ and $f_{\max}\phi$ belong to $U_{n+1}^{h^{\im}}$. First we show that if \(i<(n)f_{\max}\), then $(i)h=(i+1)h$ implies that $(i)\phi=(i+1)\phi$.
Observe that for all \(i\in \N\) we have
\begin{align*}
&(i)h=(i+1)h \iff \exists m \hbox{ such that } (i, i+1\in (x_m)h^{-1}) \iff\\
& \iff \exists m \hbox{ such that } ((m)f_{\min} \leq i< (m)f_{\max}).
\end{align*}
Hence if \(i< (n)f_{\max}\) and \((i)h=(i+1)h\), then there is \(m\in\N\) with \((m)f_{\min} \leq i< (m)f_{\max}\), and
$$(i)\phi\leq (i+1)\phi\leq (m)f_{\max}\phi=(m)h^{\im} =(m)f_{\min}\phi\leq (i)\phi,$$
where the equalities hold since \((m)f_{\min}\leq i<(n)f_{\max}\) implies that \(m<n\), and from above $f_{\min}\phi, f_{\max}\phi\in U_{n+1}^{h^{\im}}$. So \((i)\phi=(i+1)\phi\).
Hence $(i)h=(i+1)h$ implies that $(i)\phi=(i+1)\phi$ for all  \(i<(n)f_{\max}\).

We show by induction that  \((i)\phi=(i)h\) for all \(i<n\).
Since \((0)f_{\min}=0\), we get that $$(0)\phi=(0)f_{\min}\phi=(0)h^{\im}=x_0=(0)h.$$
Suppose that $\phi{\restriction}_i=h{\restriction}_i$ for some \(0<i<n\). We show that \((i)\phi=(i)h\).
There are two cases to consider:
\begin{enumerate}
 \item  \((i)h=(i-1)h\);
 \item  \((i)h>(i-1)h\).
\end{enumerate}

 (1)  Since \(i-1<n\leq (n)f_{\max}\), the arguments above imply that $(i)\phi=(i-1)\phi$.  By the inductive assumption, $(i-1)h=(i-1)\phi$. Thus $(i)\phi=(i)h$.

 (2) In this case \(i=\min((i)hh^{-1}\)). So there is \(x\in \N\setminus \{0\}\) such that \((x)f_{\min}=i\). Since $f_{\min}$ is an order preserving injection, we have \(x\leq i<n\). Then, taking into account that $\phi\in W$, we obtain the following:
\[(i)h=(x)f_{\min}h=(x)h^{\im}=(x)f_{\min}\phi=(i)\phi.\]
Hence $W\subseteq U_n^h$, witnessing that $U_n^h$ is a neighborhood of $h$ with respect to the semigroup Zariski topology, as required.
\end{proof}

%A semigroup polynomial $f:S\rightarrow S$, $f(x)=a_0xa_1\cdots xa_n$

\begin{remark}\label{magenta2}
The above proof shows that, in fact, the sets $$\{x\in \operatorname{End}^\infty(\N,\leq): a_0xa_1\neq b_0xb_1\},$$ where $a_0,a_1,b_0,b_1\in \operatorname{End}^\infty(\N,\leq)$ generate the pointwise topology on $\operatorname{End}^\infty(\N,\leq)$. 
 More concretely, using \(f_{\max}, f_{\min}, g, \iota_i\) as defined in the above proof, we have shown
 \[h\in\makeset{k\in \operatorname{End}^\infty(\N,\leq)}{for all \(i<n+2\) and \(f\in \{f_{\max}, f_{\min}\}\) we have \(f kg \iota_i\neq f kg \iota_{i+1}\),\\ \(\iota_i fkg \neq \iota_{i+1} fkg\),  and \(fk\iota_{(i)h^{\im}}\neq fk\iota_{(i)h^{\im}+1}\)
 } \subseteq U_{n}^{h}\]
 for all \(h\in\operatorname{End}^\infty(\N,\leq) \) and \(n\in \N\).
\end{remark}

\begin{definition}
    We define $\operatorname{End}^\infty(\Z,\leq)$ to be the submonoid of $\operatorname{End}(\Z,\leq)$ consisting of those maps whose image neither contains a minimum element nor a maximum element.
\end{definition}

We will need to show that both \(\operatorname{End}^\infty(\mathbb{N},\leq)\) and \(\operatorname{End}^\infty(\mathbb{Z},\leq)\) have property \(\mathbb{XX}\) with respect to themselves.  The following lemma is the first step in this direction.

\begin{lemma}\label{strong surjective}
Let \(X\) be a  subset of \((\Z,\leq)\).
When giving \(\Xendd\) the pointwise topology, the following assertions hold: 
\begin{enumerate}
    \item for each surjective endomorphism $f$, the right shift $r_f$ is open;
    \item for each injective endomorphism $g$, the left shift $l_g$ is open.
\end{enumerate}
\end{lemma}
\begin{proof}
If $X$ is finite, then there is nothing to show as \(\Xendd\) is discrete, and thus each self-map of $\Xendd$ is open and closed. 

Assume that $X$ is infinite.
For all \(a,b\in X\) with \(a\leq b\) and \(f\in  \Xendd\), we define \(f_{a,b}\) to be the restriction of \(f\) to the set \(\makeset{x\in X}{\(a\leq x \leq b\)}\).
We then define
\[U^f_{a, b}:=\makeset{h\in \Xendd}{\(f_{a,b} \subseteq h\)}.\]
These sets form a basis for the topology on \(\Xendd\).

(1) Let \(f\in \Xendd\) be surjective, \(g\in \Xendd\) and $a,b\in X$ with $a\leq b$. 
In order to show that the right shift $r_f$ is open, it suffices to check that \(U_{a, b}^gf=U_{a,b}^{gf}\). The inclusion \(U_{a,b}^gf\subseteq U_{a,b}^{gf}\) is obvious. 
To show the converse inclusion fix any $w\in U_{a,b}^{gf}$. 
Note that for all \(k\in X\), the set \(I_k:=((k)w)f^{-1}\) is a nonempty convex subset of \(X\). Moreover the sets \(I_k\) are pairwise disjoint, and \(i_j<i_k\) whenever \(i_j\in I_j\), \(i_k\in I_k\) and \(j<k\). Since $(k)w=(k)gf$ providing \(a\leq k \leq b\), we have that \((k)g\in I_k\) for all \(a\leq k \leq b\).
For each \(k\in X\), the set \(\im(g_{a,b})\cap I_k\) is finite (possibly empty). So we can choose for each \(k\in X\) elements \(c^k_{\max}, c^k_{\min}\in I_k\) such that \(c^k_{\max}\) is not lower than any element of \(\im(g_{a,b})\cap I_k\), \((c^k_{\min}\) is no larger than any element of \(\im(g_{a,b})\cap I_k\), and \(c^k_{\max}\geq  c^k_{\min}\). Note that if $\im(g_{a,b})\cap I_k=\emptyset$, then  \(c^k_{\max}\geq  c^k_{\min}\) is the only condition  \(c^k_{\max}\) and \(c^k_{\min}\) are obliged to satisfy.
We define $h\in U_{a,b}^g$ as follows:
\[(k)h:=\begin{cases}
    c^k_{\min} & k< a\\
    c^k_{\max} & k> b\\
    (k)g_{a,b} & a\leq k\leq b.
\end{cases}\]
As \((k)h\in I_k\) for all \(k\in X\), it follows that \(hf=w\). 
In particular \(w=hf\in U^{g}_{a,b}f\) as required.

% the element of \(((k)w)f^{-1}\) with largest absolute value whenever \(|k|\geq n\). As \(w\in U^{gf}?n\), it follows that \(hf=w\). Moreocer 

% as follow
% \[(k)h=\begin{cases}
%     (k)g & |k|< n\\
%     ((m)w)f^{-1} &|k|\geq n
% \end{cases}\]

% $(k)h=(k)g$ for all $-n<k<n$, $(m)h=\max\big(((m)w)f^{-1}\big)$ for all $m\geq n$ and $(m)h=\min\big(((m)w)f^{-1}\big)$ for all $m\leq -n$. 

% Since the function $f$ is surjective, the map $h$ is well defined.
% It is straightforward to check that $hf=w$. We need to show that $h\in U_n^g$. Since $(k)h=(k)g$ for all \(-n<k<n\), it suffices to check that $h\in\Xendd$. Since $w,f\in\Xendd$, we get that  $$\max\big(((m)w)f^{-1}\big)\leq \max\big(((m+1)w)f^{-1}\big),$$ 
% $$\min\big(((m)w)f^{-1}\big)\geq \min\big(((m-1)w)f^{-1}\big),$$ 
% witnessing that $(m)h\leq (m+1)h$ for all $m$ such that $|m|\geq n$. 
% It remains to check that $(n)h\geq (n-1)h$ and $(-n)h\leq (-(n-1))h$. Taking into account that $(n-1)w=(n-1)hf$ and \((-(n-1))w=(-(n-1))hf\), we get $$(n)h=\max\big(((n)w)f^{-1}\big)\geq \max\big(((n-1)w)f^{-1}\big)=\max\big(((n-1)hf)f^{-1}\big)\geq (n-1)h,$$
% $$(-n)h=\min\big(((-n)w)f^{-1}\big)\leq \min\big(((-(n-1))w)f^{-1}\big)=\min\big(((-(n-1))hf)f^{-1}\big)\leq (n-1)h.$$
% Hence $h\in U_n^g$ and $U_n^gf=U_n^{gf}$, as required.

(2) Let $g, f\in\Xendd$ be such that \(g\) is an injection. We only need to show that the image of each neighborhood of \(f\) under \(l_g\) is a neighborhood of \(gf\).
Let \(a,b\in X\) be such that 
\begin{enumerate}
    \item \(a\leq b\),
    \item either \(a\in \im(g)\) or \(a=\min(X)\),
    \item either \(b\in \im(g)\) or \(b=\max(X)\).
\end{enumerate}
We show that \(gU^f_{a,b}\) is a neighborhood of \(gf\). This is sufficient as the family \[\{U_{a,b}^f: a,b \hbox{ satisfy conditions } (1)\hbox{--}(3)\}\] forms a neighbourhood base at \(f\).
Let \(a',b'\in X\) be such that whenever \(a\leq (k)g \leq b\), we have \(a'\leq k \leq b'\) (these exist as \(g\) is injective and there are only finitely many such \((k)g\)).
We show that \(U^{gf}_{a',b'} \subseteq gU^f_{a,b}\).
Let \(w\in U^{gf}_{a',b'}\) and consider the order preserving map \(g^{-1}w:\im(g) \to X\).
If \(h\in X^X\), then \(gh=w \iff h\supseteq g^{-1}w\). 
To conclude that \(w\in gU^f_{a,b}\), we need to show that \(g^{-1}w\cup f_{a,b}\) has an extension in \(\operatorname{End}(X)\).
It is easy to verify that every order preserving map between subsets of \(X\) has an extension to an element of \(\operatorname{End}(\Z,\leq)\) (in particular to an element of \(\operatorname{End}(X,\leq)\)).
Thus we only need to show that \(p:=g^{-1}w\cup f_{a,b}\) is an order preserving partial function.

We first show that \(p\) is a partial function. By the choice of \(a',b'\), if \(i\in \im(g)\) and \(a\leq i\leq b\) then \(a'\leq (i)g^{-1}\leq b'\). 
Thus \(((i)g^{-1})w=((i)g^{-1})gf=(i)f\) witnesses that \(g^{-1}w\) and \(f_{a,b}\) agree 
on $\dom(g^{-1}w)\cap \{x\in X: a\leq x\leq b\}=\im(g)\cap \{x\in X: a\leq x\leq b\}$.
Hence \(p\) is a partial function.

It remains only to show that \(p\) is order preserving. Let \(k,k^+\in \dom(p)\) be arbitrary such that \(k^+\) is the successor of \(k \) in the poset \(\dom(p)\). It suffices to show that \((k)p\leq (k^+)p\). There are 3 cases to consider:
\begin{enumerate}[\rm(a)]
    \item \(k<a\). As \(k^+\leq a\) and \(k^+\) is not minimal in \(X\), it follows from the choice of \(a\) (see condition (2)) that \(a\in \im(g)\).
    As \(k,k^+\in \dom(p)\), it follows that \(k,k^+\in \dom(g^{-1}w)=\im(g)\).
    \item \(a\leq k<b\). In this case \(k,k^+\in \dom(f_{a,b})\).
    \item \(b\leq k\). As \(k\geq b\) and \(k\) is not maximal in \(X\), it follows from the choice of \(b\) (see condition (3)) that \(b\in \im(g)\).
    As \(k,k^+\in \dom(p)\), it follows that \(k,k^+\in \dom(g^{-1}w)=\im(g)\).
\end{enumerate}

We showed above that in each case the points \(k,k^+\) belong to the domain of a common order preserving map contained in \(p\), which implies that \((k)p\leq (k^+)p\).
\end{proof}

\begin{corollary}\label{open D1}
Let \(X\in \{\N,\Z\}\). When giving $\operatorname{End}^\infty(X,\leq)$ the pointwise topology, we have that 
\begin{enumerate}
    \item for each surjective endomorphism $f$, the right shift $r_f: \operatorname{End}^\infty(X,\leq)\rightarrow \operatorname{End}^\infty(X,\leq)$ is open;
    \item for each injective endomorphism $g$, the left shift $l_g: \operatorname{End}^\infty(X,\leq)\rightarrow \operatorname{End}^\infty(X,\leq)$ is open.
\end{enumerate}
\end{corollary}
\begin{proof}
Note that $\operatorname{End}(X,\leq)\setminus \operatorname{End}^\infty(X,\leq)$ is a two-sided ideal and \(f,g\in \operatorname{End}^{\infty}(X,\leq)\).
By Lemma~\ref{strong surjective}, the maps \(r_f,l_g:\operatorname{End}(X,\leq)\to \operatorname{End}(X,\leq)\) are open. Let \(U\subseteq \operatorname{End}(X,\leq)\) be open. It is straightforward to check  that
\[(U)r_f\cap \operatorname{End}^{\infty}(X,\leq)=(U\cap \operatorname{End}^{\infty}(X,\leq))r_f.\]
Thus $(U\cap \operatorname{End}^{\infty}(X,\leq))r_f$ is open in \(\operatorname{End}^{\infty}(X,\leq)\) implying that \(r_f\) is an open map from \(\operatorname{End}^{\infty}(X,\leq)\) to itself. Similarly \(l_g\) is also an open mapping.
\end{proof}

The following algebraic fact will help us to establish that the monoids $\operatorname{End}^\infty(\N,\leq)$ and $\operatorname{End}^\infty(\Z,\leq)$ endowed with the pointwise topology have property $\mathbb{XX}$ with respect to themselves.

\begin{lemma}\label{term}
Let \(X\in \{\N,\Z\}\).    For any $f,g\in \operatorname{End}^\infty(X,\leq)$, there exist an injection $\phi\in \operatorname{End}^\infty(X,\leq)$ and a surjection $\psi\in \operatorname{End}^\infty(X,\leq)$ such that $\phi g \psi=f$.
\end{lemma}

 \begin{proof}
Let $\phi\in \operatorname{End}(X,\leq)$ be an arbitrary endomorphism such that 
\begin{enumerate}
    \item $(0)\phi g \geq (0)f$,
    \item $(n+1)\phi g-(n)\phi g> (n+1)f-(n)f$  for all $n\in X$,
\end{enumerate}
The function $\phi$ can be constructed inductively exploiting the fact that $g\in \operatorname{End}^\infty(X,\leq)$.
Note that in particular \(\phi\) and \(\phi g\) are injective and order preserving.
 
 % Let \(\psi':=(\phi g)^{-1}f\), and note that \((\phi g)\psi'=f\) so we need only show that there is a surjection \(\psi\in \Zendd\) with \(\psi'\subseteq \psi\). 

For all \(n\in X\), let 
\[\psi_n:\makeset{x\in X}{\((n)\phi g\leq x\leq (n+1)\phi g\)}\rightarrow \makeset{x\in X}{\((n)f\leq x\leq (n+1)f\)}\] be an order-preserving surjection.
Note that we must have 
%$$((n)\phi g)\psi_n=(n)f,$$
%and
\[((n+1)\phi g)\psi_n = (n+1)f=((n+1)\phi g)\psi_{n+1}\] for all \(n\in X\).
In particular, \(\psi_n\) and \(\psi_{n+1}\) agree on the unique point in both of their domains.
Thus
\[\psi:=\bigcup_{n\in X} \psi_n\]
is an order-preserving partial function. 
The domain of \(\phi\) is the set of elements of \(X\) which are between points in the image of \(\phi g\), and the image of \(\phi\) is the set of points between points in the image of \(f\).
If \(X=\Z\), then \(\psi\) is a surjective fully defined function, and if \(X=\N\) then, since \((0)\phi g\geq (0)f\), we can extend \(\psi\) to a surjective fully defined function.
Moreover, by construction we have \(\phi g \psi =f\) as required.
\end{proof}

The following fact follows from Corollary~\ref{open D1} and Lemma~\ref{term}.

\begin{proposition}\label{XXD1}
 The monoids $\operatorname{End}^\infty(\N,\leq)$ and $\operatorname{End}^\infty(\Z,\leq)$ endowed with the pointwise topology have property $\mathbb{XX}$ with respect to themselves.  
\end{proposition}

We can now prove the second and final item of \cref{thm:endNi}.

\begin{proposition}\label{maximum}
      The pointwise topology is the unique Polish semigroup topology on $\operatorname{End}^\infty(\N,\leq)$. 
\end{proposition}
\begin{proof}
Let \(\Tau\) be a Polish semigroup topology on \(\operatorname{End}^\infty(\N,\leq)\). As \(\Tau\) is a Hausdorff semigroup topology on \(\operatorname{End}^\infty(\N,\leq)\), \cite[Proposition 2.1]{main} implies that \(\Tau\) contains the semigroup Zariski topology on \(\operatorname{End}^\infty(\N,\leq)\). By \cref{magenta}, it follows that \(\Tau\) is contains the pointwise topology. Since $\operatorname{End}(\N,\leq)$ equipped with the pointwise topology is Polish and $\operatorname{End}(\N,\leq)\backslash \operatorname{End}^\infty(\N,\leq)$ is countable, it follows that $\operatorname{End}^\infty(\N,\leq)$ is Polish as well.
Thus \(\Tau\) is equal to the pointwise topology by \cref{cor:XX_main} and \cref{XXD1}.
    \end{proof}

\section{\(\operatorname{End}(\N, \leq)\) and $\Zendd$ }

In this section we prove \cref{thm:endN} and \cref{thm:endZ}. %Item 1 of each of these gives a coarsest Hausdorff semigroup topology on the corresponding semigroup.
The following result was proven in~\cite[Lemma 5.1]{main} and will be used in the proof of \cref{zariski_order} to establish the coarsest Hausdorff semigroup topology on the monoids 
\(\operatorname{End}(\N, \leq)\) and $\Zendd$.
\begin{lemma}
  \label{thm-elliott-1}
  Let $X$ be an infinite set, and let $S$ be a subsemigroup of partial transformation monoid $P_X$ such that
  $S$ contains all of the constant  transformations (defined everywhere on
  $X$), and for every $x\in X$
  there exists $f_{x}\in S$ such that $(x) f_{x} ^ {-1} = \{x\}$ and $(X)f_{x}$
  is finite.  If $\mathcal{T}$ is a shift-continuous topology on
  $S$, then the following are equivalent:
  \begin{enumerate}
    \item \label{thm-elliott-1-i} 
    $\mathcal{T}$ is Hausdorff;
    \item \label{thm-elliott-1-ii} 
    $\mathcal{T}$ is $T_1$;
    \item \label{thm-elliott-1-iii} 
    $\{f\in S: (y, z) \in f\}$ and $\{f \in S :
            t\not\in\dom(f)\}$ are open with respect to
          $\mathcal{T}$ for all $y, z, t \in X$;
    \item \label{thm-elliott-1-iv} 
    $\{f\in S: (y,z)\in f\}$ and $\{f \in S :
            t\not\in\dom(f)\}$ are closed with respect to
          $\mathcal{T}$ for all $y, z, t \in X$.
  \end{enumerate}
\end{lemma}

\begin{proposition}\label{zariski_order}
    If \((X, \leq)\) is a totally ordered set, then the smallest \(T_1\) shift-continuous topology on \(\operatorname{End}(X, \leq)\) is the pointwise topology. 
\end{proposition}
\begin{proof}
    First, if \(X\) is finite then this is obvious. 

    Otherwise  \(\operatorname{End}(X, \leq)\) is a subsemigroup of  \(P_X\) where \(X\) is infinite. So we need only check the conditions of Lemma~\ref{thm-elliott-1}. The constant transformations clearly belong to \(\operatorname{End}(X, \leq)\).
    Moreover if \(x\in X\) then we can define \(f_x\) as follows. If \(x\) is not the largest element of \(X\) then define \(x_b\) to be a fixed larger element (otherwise leave it undefined). If \(x\) is not the smallest element of \(X\) then define \(x_s\) to be a fixed smaller element (otherwise leave it undefined). Then define
    \[(y)f_x=\begin{cases}
        x &y=x;\\
        x_s & y<x;\\
        x_b & y>x.
    \end{cases}\]    
\end{proof}

Since the semigroup Zariski topology is always \(T_1\), shift-continuous, and contained in every Hausdorff semigroup topology, \cref{zariski_order} implies the following.
\begin{corollary}
 If \((X, \leq)\) is a totally ordered set, then the semigroup Zariski topology on \(\operatorname{End}(X, \leq)\) is the pointwise topology.  
\end{corollary}

The following two definitions are used to construct Polish topologies for \(\operatorname{End}(\Z,\leq)\) and, in particular, to prove items 2 and 3 of \cref{thm:endZ}. 

\begin{definition}
Consider the  following sets.
For $x, y\in\Z$, let 
\begin{align*}
    B_x^-&=\makeset{f\in \Zendd}{\(x\leq\min(\im(f))\)},\quad\hspace{8mm} B_x^+=\makeset{f\in \Zendd}{\(x\geq\max(\im(f))\)},\\
    C^-&=\makeset{f\in \Zendd}{\(f\) is unbounded below},\ C^+=\makeset{f\in \Zendd}{\(f\) is unbounded above},\\
    U_{x,y}&=\makeset{f\in \Zendd}{\((x, y)\in f\)}.
\end{align*}
\end{definition}

\begin{definition}\label{Def71}
  For \(n\in \N\), let \(\Tau_n\) be the least topology on $\Zendd$ such that the following sets are simultaneously open and closed
  \begin{enumerate}[\rm (a)]
      \item \(U_{x,y}\) for \(x,y\in \Z\);
    \item \(C^{+},C^{-}\);
     \item \(B_{x}^{+}\cap C^{-}\), \(B_{x}^{-}\cap  C^{+}\) for \(x\in \Z\);
    \item the set $D_n$ of elements with image size strictly less than \(n\);
     \item each singleton with finite image size at least \(n\).
  \end{enumerate}
\end{definition}

\begin{lemma}\label{TODO}
    The topologies \(\Tau_n\) are all Polish.
\end{lemma}
\begin{proof}
Let \(n\in \N\).
We show \(\Tau_n\) is Polish by building a sequence of topologies on \(\Zendd\) starting with the pointwise topology (which is Polish), and at each stage we add countably many open sets which are closed in the previous topology and apply \cref{PolishBuilding} to conclude the new topology is Polish.
At each step we make closed sets open (and hence clopen) and each of these sets belongs to \(\Tau_n\).
The sequence will conclude with a topology containing all the sets defining \(\Tau_n\) so the result follows.

\begin{enumerate}
    \item 
     The pointwise topology (generated by the sets \(U_{x,y}\)) is Polish.
     \item Next add the singletons \(\makeset{\{f\}}{\(f\) has finite image size at least \(n\)}\).
     \item  Next add the complement of the open set \(\makeset{f\in \Zendd}{\(f\) has finite image size at least \(n\)}\). Note that $C^+ \cup C^-=\{f\in \Zendd: |\im(f)|=\aleph_0\}$. The following formula gives another representation of the set $C^+ \cup C^-$ which will be used to show that $C^+ \cup C^-$ is open. 
     \[C^+ \cup C^-=\makeset{f\in \Zendd}{\(|\im(f)|\geq n\)  and (\(|\im(f)|=\aleph_0\) or \(|\im(f)|<n\))}.\]
     Observe that $\{f\in \Zendd: |\im(f)|\geq n\}$ is open, as it is open in the pointwise topology. Also note that $\{f\in \Zendd: |\im(f)|=\aleph_0\ \hbox{ or } \im(f)|<n\}$ is open, as it equals the complement of the set \(\makeset{f\in \Zendd}{\(f\) has finite image size at least \(n\)}\), which we recently announced open.
     
     \item Next add the complement of the set \(C^+\cup C^-\). Since the sets $B_x^+$ and $B_x^-$ are closed with respect to the pointwise topology, we get that \(B_{x}^{+}\cap (C^{+}\cup C^{-})=B_{x}^{+}\cap C^{-}\) and \(B_{x}^{-}\cap (C^{+}\cup C^{-})=B_{x}^{-}\cap C^{+}\) are now closed for all \(x\in \Z\).
     \item  Next add the sets \(B_{x}^{+}\cap C^{-}\), \(B_{x}^{-}\cap C^{+}\) for all \(x\in \Z\).
     \item At this point the sets \(B_{x}^+, B_{x}^-\) have open intersection with the clopen set \(C^+\cup C^-\). Thus removing any family of these sets from \(C^+\cup C^-\) yields a closed set.
    As such, we can finally add the sets \(C^{+}=(C^+ \cup C^-)\setminus \bigcup_{x\in \Z}{B^{+}_x}\) and \(C^{-}=(C^+ \cup C^-)\setminus \bigcup_{x\in \Z}{B^{-}_x}\). \qedhere
\end{enumerate}
\end{proof}

\begin{lemma}\label{*}
  For all \(n\in \N\),  \((\Zendd, \Tau_n)\) is a Polish topological semigroup.
\end{lemma}
\begin{proof}
Let \(n\in \N\) be fixed.
By Lemma~\ref{TODO} the topology \(\Tau_n\) is Polish.
We next show that the preimages under the semigroup operation \(\circ\) on $\Zendd$ of the sets from \cref{Def71} are open. This will imply that  \(\Tau_n\) is a semigroup topology. First note that the pointwise topology is a semigroup topology. So for every $x,y\in\Z$ the preimage under multiplication $(U_{x,y})\circ^{-1}=\{(f,g):fg\in U_{x,y}\}$ is open.
We have now shown that the sets from part (a) of \cref{Def71} have open preimages.

Next consider the discrete topological semigroup \(S:=\{\operatorname{true}, \operatorname{false}\}{\times}\{\operatorname{true}, \operatorname{false}\}\) with the operation of pointwise ``and". 
The map \(\phi: \Zendd \to S\) defined by 
\[(s)\phi=(s\text{ is unbounded below}, s\text{ is unbounded above})\]
is a semigroup homomorphism. The preimage of the discrete topology under \(\phi\) is generated by \(C^+, C^-\) and their complements. Thus the topology generated by these sets and their complements is also a semigroup topology on \(\Zendd\). Thus the preimages under $\circ$ of \(C^+,C^-\) and their complements are open with respect to \(\Tau_n\) (or indeed any topology containing \(C^+,C^-\) and their complements). We have now shown that the sets from part (b) of \cref{Def71} have open preimages.
 
Let \(x\in \mathbb{Z}\). Since \(B_x^+\) is closed in the pointwise topology (and hence has closed preimage) and the preimage of \(C^+\) is closed, it follows that the preimage of \(B_x^+\cap C^+\) is also closed. 
That is to say that the preimage of the complement of \(B_x^-\cap C^{+}\) under \(\circ\) is open. We now show that the preimage of \(B_x^-\cap C^{+}\) is open.
If \(f,g\in \Zendd\) and \(\im(fg)\) is bounded below by \(x\) and unbounded above, then \(\im(f), \im(g)\) are unbounded above and at least one of the following conditions holds
    \begin{itemize}
        \item  \(\min(\im(f))\) is defined and \((\min(\im(f)))g \geq x\);
        \item \(\min(\im(g))\) is defined and is at least \(x\).
    \end{itemize}
    In the first case if \((z)f=y=\min(\im(f))\), then  \[f\circ g\in ((B_{y}^-\cap U_{z,y}\cap C^+)\times (U_{y,(y)g}\cap C^+))\circ \subseteq B_x^-\cap C^+.\] In the second case \[f \circ g\in (C^+ \times (B_{x}^-\cap C^+))\circ \subseteq B_x^-\cap C^+.\]
    The formulas above imply that $(B_x^-\cap C^+)\circ^{-1}$ is open.
    The fact that the set \((B_x^+\cap C^{-})\circ^{-1}\) is clopen follows by a symmetric argument. We have now shown that the sets from part (c) of \cref{Def71} have open preimages.

    We need to show that the set 
    \[Q_n:=\makeset{(a, b)\in \Zendd{\times}\Zendd}{\(|\im(a\circ b)|<n\)}=(D_n)\circ^{-1}\] is open (see item (d) of \cref{Def71}).
    Let \((a, b)\in Q_n\).
    We are going to show that \(Q_n\) is a neighborhood of \((a, b)\). First assume that \(a\) has image size at most \(n-1\). Then $D_n$ is an open neighborhood of $a$ in $\Tau_n$. Since $D_n$ is a two-sided ideal in $\Zendd$, it is easy to see that in this case $(a,b)\in D_n{\times}\Zendd\subseteq Q_n$. If \(b\) has image size at most \(n-1\), then $(a,b) \in \Zendd{\times} D_n\subseteq Q_n$. Hence if $a$ or $b$ belongs to $D_n$, then $Q_n$ is a neighborhood of $(a,b)$. 

    To prove the continuity of \(\circ\) at the remaining elements of \(Q_n\) as well as on the preimages of the singletons from \(\makeset{\{f\}}{\(n\leq|\im(f)|<\aleph_0\)}\) (see item (e) of \cref{Def71}), we show that if \(a,b\in \Zendd\) have image size at least \(n\) and \(ab\) has finite image then \((ab)\circ^{-1}\) is open. This will show that the preimages of the sets from items (d) and (e) in \cref{Def71} are open. 
    
    There are 16 cases (many of which will be covered simultaneously) depending on the value of \(((a)\phi, (b)\phi)\). For shorter notation, let \(a_0,a_1,b_0,b_1\in\{\operatorname{true},\operatorname{false}\}\) be such that \[((a)\phi, (b)\phi)=((a_0,a_1),(b_0,b_1)).\] Below we outline five arguments which cover all the 16 cases.
    
    \begin{enumerate}[\rm(i)]
        \item If \(a_0=b_0=\operatorname{true}\) or \(a_1=b_1=\operatorname{true}\), then as \((a)\phi(b)\phi = (ab)\phi = (\operatorname{false}, \operatorname{false})\), we reach a contradiction. So these cases do not occur (this covers 7 cases);
         \item  If \(a_0=a_1=\operatorname{false}\), then \(a\) is isolated. Let \(U\) be the set of all endomorphisms which agree with \(b\) on the finitely many points in the image of \(a\). Since the pointwise topology on $\Zendd$ is contained in $\Tau_n$, the set $U$ is an open neighborhood of $b$. Then \((\{a\}\times U)\circ\subseteq \{ab\}\) (we have now covered 11 cases); 
        \item  If \(b_0=b_1=\operatorname{false}\), then \(b\) is isolated.
        Moreover there is \(k\in \N\) such that \(b\) is constant on each of the sets \(\makeset{m\in \Z}{\(m\geq k\)}\) and \(\makeset{m\in \Z}{\(m\leq -k\)}\).
        Define a neighbourhood \(U\) of \(a\) according to the following cases:
        \begin{enumerate}
            \item if \(a_0=a_1=\text{false}\), then \(U=\{a\}\);
            \item if \(a_0=a_1=\text{true}\), then let \(l,r\in \Z\) be such that \((l)a<-k\leq k<(r)a\). Since the pointwise topology on $\Zendd$ is contained in $\Tau_n$, the set $U$ consisting of all endomorphisms which agree with \(a\) on \(\{l,l+1,\ldots,r-1,r\}\) is open;
            \item if \(a_0=\text{true}\) and \(a_1=\text{false}\), then let \(l,r\in \Z\) be such that \((l)a<-k\) and \((r)a=\max(\im(a))\).
            Let \(U\) be the set of endomorphisms which are unbounded below, have \(\max((r)a)\) as the largest point in their image and agree with \(a\) on \(\{l,l+1,\ldots,r-1,r\}\). Since $U=C^-\cap B_{(r)a}^+\cap \bigcap_{l\leq x\leq r}U_{x,(x)a}$, it is open;
            \item if \(a_0=\text{false}\) and \(a_1=\text{true}\), then define \(U\) analogously to the previous case.
        \end{enumerate}
        In all these subcases, the elements of \(U\) agree with \(a\) on all integers which \(a\) maps into \([-k,k]\). It follows that \((U\times \{b\})\circ\subseteq\{ab\}\) (we have now covered 14 cases);
       \item  Assume that  \(a_1=b_0=\operatorname{false}\) and \(a_0=b_1=\operatorname{true}\). Let $x=\max(\im(a))$ and $y=\min(\im(b))$. Fix $k\in\N$ large enough that \((k)a=x\) and \((-k)ab=y\). 
       Let \(W_a\) be the set of maps which are unbounded below, agree with \(a\) on the interval from \(-k\) to \(k\) and which have \(x\) as their largest image value.
       Let \(W_b\) be the set of maps which are unbounded above, agree with \(b\) on the interval from \((-k)a\) to \(x\) and have \(y\) as their smallest image value. It follows that \((W_a\times W_b)\circ\subseteq \{ab\}\) as required (we have now covered 15 cases).
     \item The last case is when  \(a_0=b_1=\operatorname{false}\) and \(a_1=b_0=\operatorname{true}\). The argument is analogous to the one from item (iv).
 \end{enumerate}
 Hence for each $n\in\N$, \(\Tau_n\) is a Polish semigroup topology on $\Zendd$.
\end{proof}

The following lemma implies that the topologies $\Tau_n$, $n\in\N$ are all different.

\begin{lemma}\label{distict polish Z}
    If \(n\in \N\) and \(f\in \Zendd\) has finite image, then \(f\) is an isolated point of \(\Tau_n\) if and only if \(|\im(f)|\geq n-1\).
    Moreover if \(S:=\makeset{f\in \Zendd}{\((x)f=-1\) for \(x<0\) and \((\N)f\subseteq \N\)}\) and \(f\in S\) has finite image, then \(f\) is an isolated point of \(\Tau_n|_S\) if and only if \(|\im(f)|\geq n-1\).
\end{lemma}
\begin{proof}
Suppose that \(f\in \Zendd\) and \(|\im(f)|\geq n-1\). If \(|\im(f)|\geq n\), then by \cref{Def71}(e), \(f\) is an isolated point in \((\Zendd,\Tau_n)\).
If \(\im(f)=n-1\), then let \(k\in \N\) be such that \((k)f=\max(\im(f))\) and \((-k)f=\min(\im(f))\). It follows from \cref{Def71}(d) that
\[\{f\}=\makeset{g\in \Zendd}{\(g\) has image size strictly less than \(n\) and \((i)g=(i)f\) whenever \(|i|\leq k\)}\]
is open with respect to \(\Tau_n\).

Suppose now that \(f\in \Zendd\) and \(|\im(f)|<n-1\).
By \cref{Def71},  the sets
\[U_k:=\makeset{g\in \Zendd}{\(|\im(g)|\leq n-1\) and \((i)g=(i)f\) whenever \(|i|\leq k\)}\]
for \(k\in \N\) form a neighbourhood base at \(f\) with respect to \(\Tau_n\).
For each \(k\in \N\) with \(k> \max(\im(f))\), let \(f_k\in \Zendd\) be defined by
\[(i)f_k=\begin{cases}
    k & i>k\\
    (i)f & i\leq k.
\end{cases}\]
As \(f_k\in U_k\backslash \{f\}\) for all \(k> \max(\im(f))\), it follows that \(f\) is not isolated with respect to $\Tau_n$. 

The arguments above imply that each $f\in S$ with $|\im(f)|\geq n-1$ is isolated in $(\Zendd,\Tau_n)$. Thus $f$ is also isolated in $S$ equipped with the subspace topology.
If \(f\in S\) is such that $|\im(f)|< n-1$, then the defined above endomorphisms \(f_k\) for $k\in \N$ belong to $U_k\cap S$.  So \(f\) is not isolated with respect to \(\Tau_n|_S\).
\end{proof}
\begin{corollary}\label{infinite polish Z}
    The monoid \(\Zendd\) admits infinitely many Polish semigroup topologies.
\end{corollary}
\begin{proof}
    By \cref{*}, each of the topologies \(\Tau_n\) makes \(\Zendd\) a Polish topological semigroup. \cref{distict polish Z} implies that these topologies are distinct for \(n\geq 2\).
\end{proof}

\begin{corollary}\label{infinite polish N}
    The monoid \(\operatorname{End}(\N,\leq)\) admits infinitely many Polish semigroup topologies.
\end{corollary}
\begin{proof}
Note that the subsemigroup \(S\) of \(\Zendd\) defined in \cref{distict polish Z} is isomorphic to \(\operatorname{End}(\N,\leq)\). The set
\[\Zendd \backslash S= \makeset{g\in \Zendd}{\((0)g <0\) or there is \(x<0\) with \((x)g\neq -1\)}\]
is open with respect to the pointwise topology.
Thus $S$ is closed in $\Zendd$ and by \cref{*}, each of the topologies \(\Tau_n|_S\) turn \(\operatorname{End}(\N,\leq)\) into a Polish topological semigroup. By \cref{distict polish Z}, these topologies are distinct for \(n\geq 3\).
\end{proof}

We now identify the largest Polish topology for each monoid.

\begin{definition}
    Let \(\Tau_{\max}\) denote the topology on \(\operatorname{End}(\N,\leq)\) generated by the pointwise topology as well as the set \(\operatorname{End}^\infty(\N,\leq)\) and all the singletons from \(\operatorname{End}(\N,\leq)\backslash\operatorname{End}^\infty(\N,\leq)\).
\end{definition}
\begin{theorem}\label{maxN}
  The finest Polish semigroup topology on \(\operatorname{End}(\N,\leq)\) is \(\Tau_{\max}\).
\end{theorem}
\begin{proof}
If we identify the subsemigroup $S$ of $\Zendd$ defined in \cref{distict polish Z}, then the topology \(\Tau_{\max}\) coincides with the topology $\Tau_0|_S$  (see \cref{infinite polish N}). Thus $\Tau_{\max}$ is a Polish semigroup topology.

Let \(\Tau\) be a Polish semigroup topology on \(\operatorname{End}(\N,\leq)\). Fix any \(U\in \Tau\). It is routine to see that \(U\backslash {\operatorname{End}^\infty(\N,\leq)} \in \Tau_{\max}\). It suffices to show that \(U\cap \operatorname{End}^\infty(\N,\leq) \in \Tau_{\max}\).
As \(\operatorname{End}^\infty(\N,\leq)\) is open with respect to \(\Tau_{\max}\), it is sufficient to show that the subspace topology on \(\operatorname{End}^\infty(\N,\leq)\) induced by \(\Tau\) is equal to the pointwise topology.

    Since the set \(\operatorname{End}^\infty(\N,\leq)\) is \(G_\delta\) in $\Endd$ with respect to \(\Tau\), we get that
    \(\Tau|_{\operatorname{End}^\infty(\N,\leq)}\) is a Polish semigroup topology.
    It follows from \cref{maximum} that \(\Tau|_{\operatorname{End}^\infty(\N,\leq)}\) is equal to the pointwise topology, as required.   
\end{proof}

\begin{theorem}\label{maxZ}
The finest Polish semigroup topology on $\Zendd$ is $\Tau_0$. 
\end{theorem}
\begin{proof}
    \cref{*} implies that \(\Tau_0\) is a Polish semigroup topology.

    Consider the partition 
    \[\mathbb P:=\{\operatorname{End}^\infty(\Z,\leq), \Zendd\backslash (C^-\cup C^+), C^+\backslash C^-,C^-\backslash C^+\}\] of \(\operatorname{End}(\Z,\leq)\).
    Note that each element of \(\mathbb P\) is a clopen subsemigroup of \(\operatorname{End}(\Z,\leq)\) with respect to \(\Tau_0\). We only need to show that if \(\Tau\) is a Polish semigroup topology on \(\operatorname{End}(\Z,\leq)\), then \(\Tau|_B\subseteq \Tau_0|_B\) for all blocks \(B\) of the partition \(\mathbb P\). 

    Consider the first block \(B_1=\operatorname{End}^\infty(\Z,\leq)\). The topology \(\Tau\) contains the pointwise topology by \cref{zariski_order}. Since 
    \[B_1=\bigcap_{n\in \N}\makeset{f\in \Zendd}{there are \(a,b\in \Z\) with \((a)f\leq -n\leq n\leq (b)f\)},\]
    we get that $B_1$ is \(G_\delta\) with respect to the pointwise topology. It follows that the topology \(\Tau|_{B_1}\) is Polish. \cref{XXD1} and \cref{cor:XX_main} imply that \(\Tau|_{B_1}\) coincides with the pointwise topology, which in turn coincides with \(\Tau_0|_{B_1}\).

    Consider the second block \(B_2=\Zendd\backslash (C^-\cup C^+)\). The topology \(\Tau_0|_{B_2}\) is the discrete topology so it follows immediately that \(\Tau|_{B_2}\subseteq \Tau_0|_{B_2}\).

    Consider the third block \(B_3=C^+\backslash C^-\). Let \(f\in B_3\) be arbitrary, let \(y=\min(\im(f))\) and let \(x\) be the largest integer with \((x)f=y\).
    Let 
    \[U:=B_{y}^-\cap U_{x,y}\cap (\bigcup_{z>y} U_{x+1,z}).\]
    The set \(U\) is an open neighbourhood of \(f\) with respect to \(\Tau_0\) and consists of those \(g\in \Zendd\) with \(y=\min(\im(g))\) and \(x\) is the largest integer with \((x)g=y\). 
    Thus we only need to show that \(\Tau|_U\subseteq \Tau_0|_U\) as this implies that the identity map from \((B_3,\Tau_0)\) to \((B_3,\Tau)\) is continuous on an open neighbourhood of \(f\). Since $f$ is chosen arbitrarily, we would get that $\Tau|_{B_3}\subseteq \Tau_0|_{B_3}$. For all \(z\in \Z\), let \(s_z\) be the element of the group of units of \(\operatorname{End}(\Z,\leq)\) which adds \(z\) to each integer.
    Then 
    \[s_{x+1}Us_{-y-1}=\makeset{g\in \Zendd}{\(\min(\im(g))=-1\), \((-1)g=-1\) and \((0)g >-1\)}\]
    is the set \(S\) from \cref{distict polish Z}. 
    As shifting by elements of the group of units in a topological semigroup is a homeomorphism, it is sufficient to show that \(\Tau|_S\subseteq \Tau_0|_S\). As \((S,\Tau_0|_S)\) is topologically isomorphic to \((\Endd,\Tau_{\max})\), the containment  \(\Tau|_S\subseteq \Tau_0|_S\) follows from \cref{maxN}.

    Finally, consider the fourth block \(B_4=C^-\backslash C^+\). Let \(\iota\) be the involution \(z\mapsto -z\) of \(\Z\). 
    Note that \(\iota\in \operatorname{Sym}(\Z)\) has order 2 and so the map \(\phi\) defined by \((f)\phi = \iota f \iota\) is an automorphism of the semigroup \(\Z^\Z\), where the semigroup operation on \(\Z^\Z\) is composition of maps. As \(\iota\) is order reversing, \((\operatorname{End}(\Z))\phi=\operatorname{End}(\Z)\).
   It is straightforward to check that \(\phi\) maps each set from \cref{Def71}(a)--(e) to another set from \cref{Def71}(a)--(e). It follows from the definition of $\Tau_0$ that the map $\phi{\restriction}_{\Zendd}$ is a homeomorphism of $(\Zendd,\Tau_0)$.
    Moreover \(\phi\) swaps the blocks \(B_4=C^-\backslash C^+\) and \(B_3=C^+\backslash C^-\).  Note that \((\Tau)\phi:=\makeset{(U)\phi}{\(U\in \Tau\)}\) is a Polish semigroup topology on $\Zendd$. 
    By the argument for \(B_3\) using \((\Tau)\phi\) as \(\Tau\), we have \((\Tau)\phi|_{B_3}\subseteq \Tau_0|_{B_3}\).
    Since  
    \[\Tau|_{B_4}\subseteq \Tau_0|_{B_4}\iff (\Tau)\phi|_{B_3}\subseteq \Tau_0|_{B_3},\]
    it follows that $\Tau|_{B_4}\subseteq \Tau_0|_{B_4}$ as required.
\end{proof}

\section{Open questions}
Figure~\ref{table-the-only} shows us that there are semigroups which admit exactly 0, exactly 1, exactly \(\aleph_0\), or exactly \(2^{\aleph_0}\) Polish semigroup topologies.
The following proposition takes this one step further.
\begin{proposition}
There is a semigroup which admits exactly \(2^{2^{\aleph_0}}\) Polish semigroup topologies.
\end{proposition}

\begin{proof}
Let $S$ be the set of reals $\mathbb R$ equipped with the left zero semigroup operation, that is, \((x,y)\mapsto x\). it is easy to check that every topology on $S$ is a semigroup topology.  Let us show that $S$ has precisely $2^{2^{\aleph_0}}$ Polish topologies. 
Let \(\Tau\) be the standard topology on \(\mathbb{R}\), and for all \(\sigma\in \Sym(\mathbb R)\) consider the topology
\[(\Tau)\sigma =\makeset{(U)\sigma}{\(U\in \Tau\)}.\]
This topology is Polish, as \(\sigma: (\mathbb R,\Tau)\rightarrow (\mathbb R,(\Tau)\sigma)\) is a homeomorphism.
Observe that for any $\sigma_1,\sigma_2\in\Sym(\mathbb R)$ we have
\[(\Tau)\sigma_1 =(\Tau)\sigma_2 \iff (\Tau)\sigma_1\sigma_2^{-1} =\Tau \iff \sigma_1\sigma_2^{-1} \text{ is a homeomorphism of }(\mathbb R,\Tau).\]
Since there are only $2^{\aleph_0}$ homeomorphisms of $(\mathbb R,\Tau)$ and $|\Sym(\mathbb R)|=2^{2^{\aleph_0}}$, it follows that there are at least \(2^{2^{\aleph_0}}\) Polish semigroup topologies topologies on $S$.

Observe that each Polish topology on $S$ is generated by a countable family of subsets of $S$. Since $S$ possesses $2^{2^{\aleph_0}}$ subsets, we have that there are at most $(2^{2^{\aleph_0}})^{\aleph_0}$ countable families of subsets of $S$. Observe that $$(2^{2^{\aleph_0}})^{\aleph_0}=2^{\max\{2^{\aleph_0}, \aleph_0\}}=2^{2^{\aleph_0}}.$$ Thus there are at most $2^{2^{\aleph_0}}$ countable families of subsets of $S$, implying that $S$ possesses at most \(2^{2^{\aleph_0}}\) distinct Polish topologies. 
\end{proof}
These observations lead us to the following question.
\begin{question}
Is it true that there is a semigroup compatible with precisely \(\kappa\) Polish semigroup topologies if and only if \(\kappa\in \{0, 1,\aleph_0,2^{\aleph_0},2^{2^{\aleph_0}}\}\)?
\end{question}

\cref{thm:endNl} suggests that the answer to the following question might be negative.

\begin{question}
    Does the semigroup $\operatorname{End}(\N,<)$ admit a finest Polish semigroup topology?
\end{question}

\cref{borelcontinuous-XX} suggests that the answer to the following question might be positive.

\begin{question}
    Do either of the Polish topological semigroups \((\Zendd,\Tau_0)\),  \((\operatorname{End}(\N,\leq),\Tau_{\max})\) have automatic continuity with respect to the class of second-countable topological semigroups? 
\end{question}

\begin{definition}
    Recall that the \(\mathscr{J}\)-preorder on a semigroup \(S\) is defined by \(s\leq_\mathscr{J} t\) if there is \(a,b\in S^1\) such that \(atb=s\). The equivalence relation on \(S\) defined by $\{(s,t)\in S{\times}S: s\leq_{\mathscr{J}}t \hbox{ and } t\leq_{\mathscr{J}}s\}$ is called the \(\mathscr{J}\) equivalence relation and the classes of this relation are called \(\mathscr{J}\)-classes.
\end{definition}

    If a semigroup \(S\) has property $\mathbb{XX}$ with respect a non-empty subset \(A\) of \(S\), then it is routine to verify that every element of \(S\) must be below every element of \(A\) in the \(\mathscr{J}\)-preorder. 
    Hence \(S\) must have a maximum \(\mathscr{J}\)-class and \(A\) must be contained within this class. 
\begin{question}
    Is there a topological semigroup \(S\) which has property $\mathbb{XX}$ with respect a non-empty subset \(A\) but not with respect to its top \(\mathscr{J}\)-class?
\end{question}

\end{document}